\documentclass[11pt,twoside]{amsart}
 
\usepackage{anysize} \marginsize{1.3in}{1.3in}{1in}{1in}
\usepackage{amsmath, amsthm, amssymb, color}
\usepackage[usenames,dvipsnames,svgnames,table]{xcolor}
\usepackage{tikz}
\usepackage{pgfplots}
\usepgfplotslibrary{fillbetween}
\pgfdeclarelayer{bg}
\pgfsetlayers{bg,main}
\usepackage{bbm}
\usepackage{stmaryrd}
\usetikzlibrary{shapes,positioning,intersections,quotes,arrows}
\usetikzlibrary{patterns}
\usepackage{graphicx}
\usepackage{caption}
\usepackage{mathtools}
\usepackage{enumerate}
\usepackage{listings}
\usepackage[all]{xypic}
\usepackage{verbatim}
\usepackage{mathrsfs}
\usepackage[linesnumbered,lined,commentsnumbered]{algorithm2e}
\usepackage{caption}

\definecolor{niceblue}{cmyk}{.93,.95,.2,.07}
\definecolor{nicegray}{cmyk}{0,0,0,.5}
\newcommand{\shiftgray}{nicegray!20}
\newcommand{\modulegray}{nicegray!50}
\newcommand{\goodred}{niceblue!40}

\newcommand{\goodwith}{.6mm}

\usepackage{stmaryrd}
\usepackage[hypertexnames=false,pdftex,backref=page,
pdfpagemode=UseNone,
breaklinks=true,
extension=pdf,
colorlinks=true,
linkcolor=MidnightBlue,
citecolor=blue,
urlcolor=blue,
]{hyperref}

\usepackage[capitalise,noabbrev]{cleveref}

\makeatletter
\newcommand{\addresseshere}{%
  \enddoc@text\let\enddoc@text\relax
}

\usepackage{makecell}
\usepackage{array}
\newcolumntype{?}{!{\vrule width 1pt}}
\newtheorem{theorem}{Theorem}
\newtheorem{proposition}[theorem]{Proposition}
\newtheorem{lemma}[theorem]{Lemma}
\newtheorem{corollary}[theorem]{Corollary}
\theoremstyle{definition}
\newtheorem{caveat}[theorem]{Caveat}
\newtheorem{definition}[theorem]{Definition}
\newtheorem{algo}[theorem]{Algorithm}

\numberwithin{theorem}{section}
\newtheorem{examplex}[theorem]{Example}
\newenvironment{example}
{\pushQED{\qed}\examplex}
{\popQED\endexamplex}

\newtheorem{remarkx}[theorem]{Remark}
\newenvironment{remark}
{\pushQED{\qed}\remarkx}
{\popQED\endremarkx}

\newcommand{\RR}{\mathbb{R}}

\newcommand{\CC}{\mathbb{C}}

\newcommand{\NN}{\mathbb{N}}
\newcommand{\KK}{\mathbb{K}}

\newcommand{\Rnn}{\mathbf{R}}

\newcommand{\param}{r}

\newcommand{\field}{\KK}

\newcommand{\basis}{\mathcal{B}}
\newcommand{\gen}{\mathfrak{g}}

\DeclareMathOperator{\diff}{d\!}

\DeclareMathOperator{\gr}{gr}
\DeclareMathOperator{\Mod}{Mod}

\DeclareMathOperator{\Tame}{Tame}
\DeclareMathOperator{\Fun}{Fun}

\DeclareMathOperator{\Vect}{Vect}
\DeclareMathOperator{\rank}{\beta_0}

\DeclareMathOperator{\id}{id}
\DeclareMathOperator{\Rec}{Rec}
\DeclareMathOperator{\image}{im}
\DeclareMathOperator{\err}{err}
\DeclareMathOperator{\Loc}{Loc}

\let\OldMarginpar\marginpar
\renewcommand{\marginpar}[1]{\OldMarginpar{\footnotesize#1}}

\title[The shift-dimension of multipersistence modules]{The shift-dimension of multipersistence modules}

\author[W. Chach\'olski]{Wojciech Chach\'olski}
\address{Wojciech Chach\'olski\\Department of Mathematics, KTH Royal Institute of Techno\-lo\-gy\\{Lindstedtsv\"agen~25, 114~28 Stockholm, Sweden}}
\email{wojtek@kth.se}

\author[R. Corbet]{Ren\'{e} Corbet}
\address{Ren\'{e} Corbet\\Department of Mathematics, KTH Royal Institute of Technology\\{Lindstedtsv\"agen~25, 114~28 Stockholm, Sweden}}
\email{corbet@kth.se}

\author[A.-L. Sattelberger]{Anna-Laura Sattelberger} 
\address{Anna-Laura Sattelberger\\Max-Planck-Institut f\"ur Mathematik in den Natur\-wis\-sen\-schaf\-ten\\Inselstra{\ss}e~22, 04103~Leipzig, Germany and Department of Mathematics, KTH Royal Institute of Technology\\Lindstedtsv\"agen~25, 114~28 Stockholm, Sweden {\em (current)}}
\email{alsat@kth.se}

\makeatletter
\@namedef{subjclassname@2020}{\textup{2020} Mathematics Subject Classification}
\makeatother


\renewenvironment{thebibliography}[1]{
  \begin{oldthebibliography}{#1}
    \setlength{\itemsep}{0.1em}
    \setlength{\parskip}{0.2em}
}
{
  \end{oldthebibliography}
}

\pgfplotsset{compat=1.18}

\begin{document}
\maketitle
\thispagestyle{empty}

\begin{abstract}
We present the shift-dimension of multipersistence modules and investigate its algebraic properties. This gives rise to a new invariant of multigraded modules over the multivariate polynomial ring arising from the hierarchical stabilization of the zeroth total multigraded Betti number. We give a fast algorithm for the computation of the shift-dimension of interval modules in the bivariate case. We construct multipersistence contours that are parameterized by multivariate functions and hence provide a large class of feature maps for machine learning tasks. 
\end{abstract}

\bigskip 

{\subsubsection*{Acknowledgments.} 
\footnotesize{We are grateful to Mats Boij for insightful discussions. 
W.C. was partially supported by VR, the Wallenberg AI, Autonomous System and Software Program (WASP) funded by Knut and Alice Wallenberg Foundation, and MultipleMS  funded by the European Union under the Horizon 2020 program, grant agreement 733161, and dBRAIN collaborative project at digital futures at KTH.
R.C. got supported by the Brummer \& Partners MathDataLab.
A.-L.S. thanks the Brummer \& Partners MathDataLab for its hospitality during her research stay at KTH Stockholm in 2021.}}

\bigskip

\setcounter{tocdepth}{1}
\tableofcontents

\section{Introduction}
Topological data analysis (TDA) reveals structures of data by topological methods. One main tool is persistent homology. To the data, one associates a filtered simplicial complex. For every natural number $n,$ one then takes the $n$-th homology with coefficients in a field $\field;$  the {\em ($n$-th) persistence module} of the data. 
The persistent homology of a one-parameter filtration is algebraically well understood;
 by a basic structure theorem from algebra, the persistence module is uniquely determined by its barcode \cite{Ghr08} in the discrete setting~\cite{EdelsbrunnerLetscherZomorodian2002,ZomorodianCarlsson2005} as well as in the real setting~\cite{BC20} under mild finiteness assumptions. Furthermore, the barcode representation of a  persistence module  is stable~\cite{stability}.
 From the barcode, one reads the birth and death times of topological features. 
The case of several parameters, i.e., the study of multifiltered simplicial complexes and their homology as introduced by Carlsson and Zomorodian in \cite{CZ09} allows the extraction of finer information from the data. As was shown in~\cite{CSV17}, the homology modules of a multifiltered simplicial complex can be obtained as the homology of a chain complex $F_1 \to F_2 \to F_3$ of free multigraded modules of finite rank by combinatorial techniques. This allows analysis of multipersistence modules using standard computer algebra software.
However, the multiparameter case is algebraically intricate. In contrast to the case of a single parameter, as pointed out in~\cite{CZ09}, the respective moduli space is not zero-dimensional.
Moreover, constructing stable and algorithmic invariants of multipersistence modules is challenging and is currently  an active branch of research in topological data analysis. 

We investigate a stable invariant of multipersistence modules that is based on the hierarchical stabilization of discrete invariants as introduced by
Scolamiero et al. in~\cite{noise} and by G\"{a}fvert and Chach\'{o}lski in~\cite{GC}. This construction 
requires the choice of a distance  between multipersistence modules, and it
turns a discrete invariant into a measurable real-valued function.
The obtained invariant is continuous with respect to the chosen distance 
on multipersistence modules and with respect to the $L_p$ distance on measurable real-valued functions. In this article, we focus on the  invariant given by the zeroth total multigraded Betti number~$\beta_0,$ often called {\em rank} of the multipersistence module in a TDA setup, 
defined as the minimal number of generators.
Its hierarchical stabilization is commonly referred to as ``stable rank'' and was introduced in~\cite{GC}.

For the hierarchical stabilization, we also need to be able to construct distances  on the space of tame persistence modules.
A way to construct a big class of metrics other than the well-known interleaving distance~\cite{lesnick2015theory} and matching distance~\cite{Biasotti2011,KerberLesnickOudot2019} is via so-called {\em persistence contours}. 
Persistence contours are  rather well-studied in the one-parameter case.  In~\cite{CR20}, for example, the authors list various contours and show how classification,
using the first one-persistence homology group,
of point processes in the unit square (sampled with respect to different probability distributions) can be  improved by choosing the right  contour. For the case of several parameters, only a small number of examples of persistence contours is known besides the standard contours. The latter are parameterized by vectors. Persistence contours can be described by simple noise systems~\cite{GC,noise}. Those are closely related to so-called ``superlinear families'' that were introduced independently in~\cite{superlinear}, a  generalization of which can be found in~\cite{BCM20}. Therein, no extensive class of parameterized such objects was given yet, but it is essential for machine learning~tasks to have such a class at one's disposal.

Multigraded modules over the polynomial ring  $\field[x_1,\ldots,x_\param]$ with the standard multigrading
can be considered naturally as multipersistence modules. 
In this article, we investigate the hierarchical stabilization of $\beta_0$ for multipersistence modules obtained in this way with respect to a broader family of contours. The obtained invariant is called {\em shift-dimension.}

The main contribution of this article is twofold. The first main contribution is the realization that the shift-dimension can be constructed and defined entirely algebraically without referring to the geometry of multipersistence modules induced by the choice of a distance. We give an explicit description of the shift-dimension and  investigate some of  its algebraic properties. We thereby establish a connection between topological data analysis and combinatorial commutative algebra.
In general, the computation of the stable rank is algorithmic but NP-hard~\cite{GC}. For the bivariate case, we present a linear-time algorithm for the computation of the shift-dimension of interval modules and therefore for the computation of the shift-dimension of quotients of monomial~ideals.

The second main contribution lies in the construction of multipersistence contours. We present a class of contours that is parameterized not only by vectors, but by multivariate functions. In doing so, we provide a large class of multipersistence~contours.

\medskip \textbf{Related work.} The investigation of invariants of multipersistence modules is currently a prominent topic in TDA. For two parameters, the Hilbert function and the rank invariant of~\cite{CZ09} are implemented in~{\tt RIVET}~\cite{rivet} 
for an interactive data analysis. It has been used in applications such as~\cite{Keller2018}. Furthermore, minimal presentations can be computed efficiently for big data sets using~{\tt RIVET}~\cite{lesnick2019computing} or~{\tt mpfree}~\cite{kerber2021fast}. 
Invariants such as the Hilbert series, associated primes, and local cohomology have been investigated~\cite{HOST17}, as well as the multirank function~\cite{thomasdiss} and the compressed multiplicity~\cite{escolar2016persistence}. 

\medskip \textbf{Outline.}
This article is organized as follows. In~\Cref{section basics}, we review basic notions about multipersistence modules, how to measure distances between them, and the hierarchical stabilization of discrete invariants. We present several new contours for the multiparameter case. \Cref{section shiftdim} \ introduces and studies the shift-dimension of tame persistence modules from an algebraic and combinatorial perspective. In particular, we focus on the investigation of (non-)additivity. We present a linear-time algorithm for the computation of the shift-dimension of interval modules in the bivariate case. In~\Cref{section algebraic}, we study the corresponding invariant of multigraded modules over the multivariate polynomial ring and investigate its algebraic properties. In \Cref{section outlook}, we give an outlook to open problems and future~work.

\section{Measuring distances between multipersistence modules}\label{section basics}
\subsection{Persistence modules}\label{sectionpersmod}
Let $(G,\ast)$ denote either  $(\NN^r,+)$ or $(\mathbb{R}_{\geq 0}^\param,+).$ We consider $G$ as a poset as follows. For $g_1,g_2\in G,$ $g_1\leq g_2$ if this holds true component-wise. Let $\field[G]$ denote the monoid ring of $G$ over the field $\field$ with its natural $G$-grading. The category $\Mod^{\gr}(\field[G])$ of $G$-graded $\field[G]$-modules is isomorphic to the thin category $\Fun \left( G,\text{Vect}_{\field}\right)$ of functors from the posetal category $G$ to the category of $\field$-vector spaces $\text{Vect}_{\field}$ as follows. To each $M=\oplus_{g\in G}M_{g}\in  \text{Ob}(\Mod^{\gr}(\field[G])),$ one associates the functor $(g \mapsto M_g)\in \text{Ob}(\Fun\left( G,\text{Vect}_{\field}\right)).$
Each graded morphism 
\[
\varphi \,=\, (\varphi_g)_{g\in G}\colon \bigoplus_{g\in G}M_{g}\rightarrow\bigoplus_{g\in G}N_{g}
\]
naturally gives rise to a natural transformation $\{\varphi_g \colon M_{g}\rightarrow N_{g}\}_{g\in G}$ of the corresponding functors. The objects of those categories are called {\em (multi-)persistence modules}. We are going to switch seamlessly between both points of view.
We denote~by  $\Tame\left( G,\text{Vect}_{\field}\right)$  the subcategory of  $\Fun\left( G,\text{Vect}_{\field}\right)$ which  corresponds to the subcategory  $\Mod^{\gr}_{\text{f.p.}}(\field[G])$ of $\Mod^{\gr}(\field[G])$ consisting of finitely presented such modules, as investigated in \cite{corbet2018representation}. 
In this article, we often consider {\em interval modules} which are defined as those elements of $\Fun(G,\Vect_{\field})$ that correspond to quotients of two monomial~ideals.  A subset $I\subseteq G$  is called {\em interval} in the poset $(G,\leq) $ if  for any $g,h\in I$ and any $f\in G$ the following two conditions hold: 
\begin{enumerate}[(i)]
\item If $g\leq f \leq h ,$ then $f\in I.$ 
\item There exist $m\in\NN$ and $g_1,\ldots,g_{2m}\in I$ such that $g\leq g_1\geq g_2\leq \cdots \geq g_{2m}\leq h.$
\end{enumerate}
A {\em free} persistence module is a module of the form $F=\oplus_{g\in G } \field(g,\bullet)^{\beta_0(g)},$ where \mbox{$\beta_0(g)\in \NN$} and $\field(g,\bullet)\in \Fun(G,\Vect_{\field})$ denotes the~functor
$$\field(g,h) \, \coloneqq \, \begin{cases}
\field & \text{if } g \leq h,\\
0 & \text{if } g \not\leq h,
\end{cases} \qquad 
\field(g, h\leq f) \, \coloneqq  \, \begin{cases}
\id_\field & \text{if }g\leq h\leq f,\\
0 & \text{otherwise.}
\end{cases}
$$ 
The natural number $\beta_0(F) \coloneqq \sum_{g\in G} \beta_0(g)$ is called the {\em rank} of the free module $F.$

\begin{definition}
Let $M \in \Fun((G,\leq),\Vect_{\field})$ be a persistence module. The {\em total zeroth multigraded Betti number }of $M,$ denoted $\beta_0(M)$ is the smallest possible rank of a free persistence module $F$ such that there exists a surjection from $F$ onto $M.$
\end{definition}
Regarding $M$ as an $\Rnn^r$-graded $\field[\Rnn^r]$-module, the number $\beta_0(M)$ is the minimal number of homogeneous generators of~$M.$
\begin{caveat}
In TDA literature, one often refers to $\beta_0$ as ``rank''. This is not consistent with standard terminology in algebra. Yet, they do coincide for free modules.  
\end{caveat}

\subsection{Distances via persistence contours} It is crucial to be able to measure distances between persistence modules. In order to {\em learn} metrics in the sense of machine learning, 
one needs a large space of such metrics parameterized by simply describable objects.
In this section, we recall a method to construct metrics on the category of tame functors, namely in terms of so-called {\em persistence contours}. Denote by $\Rnn $ the poset $\mathbb{R}_{\geq 0}$ and by $\Rnn^r_{\infty},$ where $r\in \NN,$ the poset obtained from $\Rnn^\param$ by adding a single element $\infty$ such that $x\leq \infty$ for all $x\in \Rnn^\param$ and setting $x+\infty \coloneqq \infty$ for all $x\in \Rnn^r_\infty .$

\begin{definition}\label{asvgadfgsdfg}
A {\em (multi-)persistence contour} $C$ is a functor $C\colon \Rnn^r_\infty \times \Rnn \to \Rnn^r_\infty$ such that for any $x\in \Rnn^r_\infty$ and $\varepsilon,\tau \in \Rnn,$ the following two conditions hold:
\begin{enumerate}[(i)]
	\item $x\leq C(x,\varepsilon)$ and 
	\item $C(C(x,\varepsilon),\tau)\leq C(x,\varepsilon+\tau).$
\end{enumerate}
\end{definition}
The second property in the definition indicates that a persistence contour fulfills  a {\em lax} action only, i.e., the inequality is in general not an equality.
For brevity, one often refers to persistence contours as just ``contours''.
One common example is the {\em standard contour in the direction of a vector $v\in \Rnn^\param$} defined as $C(x,\varepsilon)\coloneqq x+\varepsilon v ;$ i.e., one shifts $x$ by the vector $v$ scaled by $\varepsilon.$
An immediate generalization of that contour is to travel along a curve instead of a line as explained in~\Cref{ex:curvecontour}.

Following an immediate generalization of \cite[Section 3]{CR20}, we now recall how persistence contours give rise to pseudometrics on the category of persistence modules.

\begin{definition}
	Let $C$ be a  persistence contour, $M,N\in \Tame(\Rnn^\param,\Vect_{\field})$ tame persistence modules, and $\varepsilon\in [0,\infty).$ A morphism $f\colon M\to N$ is an {\em $\varepsilon$-equivalence} if for every $a\in [0,\infty)^\param$ s.t.\ $C(a,\varepsilon)\neq\infty,$ there is a $\field$-linear function $N(a) \to M(C(a,\varepsilon))$ making the following diagram commutative:
	$$\begin{tikzpicture}[node distance=3cm, auto]
	\node (Fa) {$M(a)$};
	\node (Ga) [below of=Fa] {$N(a)$};
	\node (FCa) [right of=Fa] {$M(C(a,\varepsilon))$};
	\node (GCa) [right of=Ga] {$N(C(a,\varepsilon))$};
	\draw[->] (Fa) to node [swap] {\tiny $f(a)$} (Ga);
	\draw[->] (Fa) to node {\tiny $M(a\!\leq\! C(a,\varepsilon))$} (FCa);
	\draw[->] (Ga) to node [swap] {\tiny $N(a\!\leq\! C(a,\varepsilon))$} (GCa);
	\draw[->] (FCa) to node {\tiny $f(C(a,\varepsilon))$} (GCa);
	\draw[->,dashed] (Ga) to node {$\exists$} (FCa);
	\end{tikzpicture}$$
\end{definition}

\begin{definition}
Two tame functors $M$ and $N$ are called {\em $\varepsilon$-equivalent} (with respect to $C$) if there is a tame functor $L$ and morphisms $f\colon M \to L \leftarrow N\colon g$ such that $f$ is an $\varepsilon_1$-equivalence, $g$ is an $\varepsilon_2$-equivalence, and $\varepsilon_1+\varepsilon_2\leq \varepsilon.$
\end{definition} 
Denote by $\mathcal{E}\coloneqq \{ \varepsilon \in [0,\infty) \mid M \text{ and } N \text{ are } \varepsilon\text{-equivalent}\}.$ Then the function
$$d_C(M,N) \,\coloneqq \,
\begin{cases}
\infty & \text{if } \mathcal{E}=\emptyset,\\
\inf(\mathcal{E}) & \text{otherwise}
\end{cases} $$
is an extended pseudometric on 
$\Tame(\Rnn^r,\Vect_{\field}).$
There is an alternative way of thinking of contours. For that, we now recall the definition of a noise system from \cite{noise}.
\begin{definition}
A {\em noise system} in $\Tame(\Rnn^r,\Vect_{\field})$ is a collection $\{\mathcal{V}_\varepsilon\}_{\varepsilon \in \Rnn}$ of sets of tame functors such that the following three conditions are satisfied:
	\begin{enumerate}[(i)]
		\item For any $\varepsilon,$ the zero functor is contained in $\mathcal{V}_{\varepsilon},$
		\item for all $0\leq \tau < \varepsilon,$ $\mathcal{V}_\tau$ is contained in $\mathcal{V}_\varepsilon,$ and
		\item if $0 \to M \to L \to N \to 0$ is a short exact sequence in  $\Tame(\Rnn^r,\Vect_{\field}),$ then
		\begin{enumerate}[(1)]
			\item if $L\in \mathcal{V}_\varepsilon,$ then also $ M,N\in \mathcal{V}_\varepsilon$ and
			\item if $M\in \mathcal{V}_\varepsilon$ and $N\in \mathcal{V}_\tau,$ then $L\in \mathcal{V}_{\varepsilon+\tau}.$
		\end{enumerate}
	\end{enumerate}
\end{definition}
Throughout this article, we will mainly focus on noise systems $\{\mathcal{V}_\varepsilon\}$ that are {\em closed under direct sums}, i.e., noise systems such that whenever $M,N\in \mathcal{V}_\varepsilon,$ also their direct sum $M\oplus N$ is in $\mathcal{V}_\varepsilon.$ 
Define $B(M,\tau)$ as 
\[
B(M,\tau)  \,\coloneqq \,  \left\{  U \mid U \text{ is a tame subfunctor of } M \ \text{such that }  M/U \in \mathcal{V}_\tau \right\}.
\]
If $\{\mathcal{V}_\varepsilon\}$ is closed under direct sums, a minimal element in $B(M,\tau)$  is unique if it exists; this element is denoted by $M[\tau]$ and is called the {\em $\tau$-shift} of $M.$ Then
$B(M,\tau)$
coincides with the set 
$$ \left\{ U \mid U\text{ is a tame subfunctor of }M \ \text{such that} \ M[\tau]\subseteq U \subseteq M \right\}.$$    

\begin{definition}[{\cite[Definition 8.2]{GC}}]\label{def simple} 
	A noise system is called {\em simple} if 
	\begin{enumerate}[(i)]
		\item it is closed under direct sums,
		\item for any tame functor $M$ and any $\tau \in \Rnn,$ $B(M,\tau)$ contains the minimal element~$M[\tau],$ and
		\item $\beta_0(M[\tau]) \leq \beta_0(M)$ for any $\tau \in \Rnn.$
	\end{enumerate}
\end{definition} 
The third condition inherently links simple noise systems to~$\beta_0.$

\begin{proposition}[{\cite[Theorem 9.6]{GC}}]
	For a persistence contour $C,$ denote by $\mathcal{V}_{C,\varepsilon}$ the set of tame functors $M$ such that $ M(x\leq C(x,\varepsilon))$ is the zero-morphism whenever $ C(x,\varepsilon ) \neq \infty.$
The function $C\mapsto \{ \mathcal{V}_{C,\varepsilon}\}_\varepsilon$ is a  bijection between the set of persistence contours and the set of simple noise systems.
\end{proposition}
Under this correspondence, for the standard contour in the direction of a vector $v,$ $M[\tau\cdot v]$ is $ (\tau\cdot v) \ast M,$ where $\ast$ denotes the induced action of the monoid $\Rnn^r$ on the monoid ring $\field[\Rnn^r]$.

\subsection{Construction of new contours}
In the setting of a single parameter, there are various contours suitable for concrete data-analytic tasks~\cite{CR20}. In this subsection, we construct several contours for the multiparameter case.

\begin{example}[Curve contour]\label{ex:curvecontour}
Consider a curve in $\mathbb{R}^\param$ that is given by a monotone non-decreasing function $I\colon \Rnn \to\mathbb{R}^\param$ for which $I(0)$ lies in one of the coordinate axes. Assume further that $\Rnn^\param$ can be covered by translations of that curve parallel to $r-1$ coordinate axes such that every point in $\Rnn^\param$ lies on precisely one of the curves.
For $x \in \Rnn^r_\infty$ and $\varepsilon \in [0,\infty),$ define $C(x,\varepsilon)$ to be the point that one obtains by traveling along the translated curve that runs $\varepsilon$-far through $x$; i.e., if $I_x$ denotes the translated curve running through $x$ at 
$\varepsilon_0,$ we set $C(x,\varepsilon)$ to be  $I_x(\varepsilon_0+\varepsilon).$
\end{example}
For some further examples of persistence contours in the one-parameter case, we refer to~\cite{CR20}. Among them is the contour of 
{\em distance type}. Fixing a vector, we can immediately generalize this contour to the multiparameter setting.

\begin{example}[Distance type in a fixed direction]\label{ex:disttype}
	Let $r\in \mathbb{N},$ $x,v\in \Rnn^r,$ and $f$ a non-negative (Lebesgue-)measurable function on $\Rnn^\param.$
	For $x\in \Rnn^r_\infty$ and $\varepsilon \in [0,\infty),$ we set $C(x,\varepsilon) \coloneqq x+\delta(x,\varepsilon)v 	$ where $\delta(x,\varepsilon) \in \mathbb{R}_{\geq 0}$ is such that
	\begin{align*}\label{everyepsilon}
	\varepsilon \,=\, \int_{L_x^{x+\delta(x,\varepsilon)v}}f \diff \mathcal{L}^r.
	\end{align*}
	 If no such $\delta(x,\varepsilon)$ exists, we set  $C(x,\varepsilon) \coloneqq \infty.$
	Here, $L_x^{x+\delta(x,\varepsilon)v}$ denotes the set 
	\[  L_x^{x+tv}  \,\coloneqq \, \left\{y\in\Rnn^r\mid y\geq x, \, y\ngeq x+\delta(x,\varepsilon)v \right\}, \] 
	to which we refer to as {\em L-shape (at $x$)}.
\end{example}
\begin{remark}
	Note that the L-shapes in~\Cref{ex:disttype} 
	can not be replaced by the rectangles of the same bounds; it would not define a lax action anymore.
\end{remark}

Applying a one-parameter contour component-wise, as in the following example, yields a multiparameter contour.
\begin{example}[Component-wise shift type]
	Let $r\in \mathbb{N},$ $f_1,\ldots,f_r\colon [0,\infty) \to [0,\infty )$ 
	be  non-negative measurable functions and $x\in \Rnn^r.$ For each $x_i,$ set $y_i$ to be the infimum of all $\tilde{y_i}\in \Rnn$ such that $x_i=\int_{0}^{\tilde{y_i}}f_i \diff \mathcal{L}^1,$ if such $\tilde{y_i}$ exists. Let $m_1,\ldots,m_r$ be super-additive non-negative functions. One obtains a contour by setting
	$$  C(x,\varepsilon) \, \coloneqq \,\begin{pmatrix}
	\int_{0}^{y_1+m_1(\varepsilon)}f_1 \diff \mathcal{L}^1\\
	\vdots\\
		\int_{0}^{y_r+m_r(\varepsilon)}f_r \diff \mathcal{L}^1
	\end{pmatrix} $$
if such $y_1,\ldots,y_r$ exist, and $C(x,\varepsilon)\coloneqq \infty$ otherwise.
\end{example}
We now introduce a generalization of the contour of shift type that is based on multivariate functions. This yields a large class of persistence contours.
\begin{example}[Multivariate shift type] \label{ex:multishiftype}
Let	$f_1,\ldots,f_r\colon [0,\infty)^r\rightarrow[0,\infty)$ be non-negative measurable  
 functions.
For $y\in \Rnn^r,$ let $\Rec(y)$ denote $[0,y_1]\times \cdots \times [0,y_r]\subseteq \Rnn^r.$
For $x\in [0,\infty)^\param$ define
$$
	B(x) \, \coloneqq \, \left\{y\mid \int_{\Rec(y)}f_1\diff \mathcal{L}^r \geq x_1,\ldots, \int_{\Rec(y)} f_r \diff\mathcal{L}^r\geq x_r\right\} 
$$ and
$	b(x) \,\coloneqq \, \inf B(x),$
	where $\inf$ denotes the meet in the poset. 
	Let $v\in\Rnn^r.$ 
	One obtains a contour by setting
	\[
	C(x,\varepsilon) \, \coloneqq  \,\sup \left( \left(\int_{\Rec(b(x)+\varepsilon v)}f_1\diff \mathcal{L}^r \, ,  \ldots, \, \int_{\Rec(b(x)+\varepsilon v)}f_r \diff\mathcal{L}^r\right), x \right)
	\]
	where $\sup$ denotes the join of the poset. 
	If~$B(x)$ is empty, set $C(x,\varepsilon)\coloneqq \infty.$ 
To verify that $C$ is a contour, one can observe that, for every $x$ and $\varepsilon$, the element $b(x)+ \varepsilon v$ belongs to $B(C(x, \varepsilon))$.  Consequently,  
 $b(C(x, \varepsilon))\leq b(x)+ \varepsilon$, and hence
 \[\int_{\Rec(C(x, \varepsilon)) + \tau v}f_i \, \diff\mathcal{L}^r \,\leq \,
\int_{\Rec(b(x)+\varepsilon v + \tau v)}f_i\, \diff\mathcal{L}^r \]
 for every $i$. This inequality implies that condition (ii) of \Cref{asvgadfgsdfg} is fulfilled.
	
The computational bottleneck is to compute~$b(x).$ This, however, can be done efficiently by an approximation scheme to arbitrary precision by studying the intersection of $B(x)$ with an equidistant grid and successively decreasing the grid length. 
\end{example}

\subsection{Hierarchical stabilization}\label{section hierarchical}
In this section, we recall from \cite{CR20,GC} the concept  and properties of the {\em hierarchical stabilization}. This process turns discrete into stable invariants as follows.
Let $T$ be a set and $d\colon T \times T \to \Rnn \cup \infty$ an extended pseudometric on $T,$ where again $\Rnn$ denotes the poset of non-negative real numbers. Let $f\colon T \to \NN$ be a discrete invariant and denote by $\mathcal{M}$ the set of Lebesgue-measurable functions from $[0,\infty) $ to $[0,\infty),$ endowed with the interleaving distance as in~\cite[Section 3]{GC}.

\begin{definition} Fix an extended pseudometric $d$ on $T.$
	The {\em hierarchical stabilization of $f$}, denoted~$\hat{f},$  is defined to be
	$$\hat{f}\colon  T \to \mathcal{M},\quad  x \mapsto \left( \tau \mapsto \min \left\{ f(y)\mid d(x,y)\leq \tau \right\} \right). $$
\end{definition}
Hence, $\hat{f}(x)\in \mathcal{M}$ is the monotone, non-increasing measurable function that maps a real number $\tau$ to the minimal value that the discrete invariant $f$ takes on a closed $\tau$-neighborhood of~$x.$ Note that $\hat{f}$ strongly depends on the chosen pseudometric on~$T.$
\begin{proposition}{{\cite[Proposition 2.2]{CR20}}} 
	For any choice of extended pseudometric on~$T,$ the function
	$\hat{f}\colon T \to \mathcal{M}$ is $1$-Lipschitz.
\end{proposition}
Therefore, the hierarchical stabilization turns discrete into stable invariants. 

\section{The shift-dimension of multipersistence modules}\label{section shiftdim}
In this section, we introduce the shift-dimension of  persistence modules and investigate its algebraic properties. We particularly focus on (non-)additivity.

\subsection{The shift-dimension and its origin}
\begin{definition}\label{def gdim}
Let $v\in\Rnn^\param,$ $M$ an $\Rnn^\param$-graded $\field[\Rnn^\param]$-module. Elements $m_1,\ldots,m_k$ of $M$ {\em $v$-generate $M$} if 
\[
v\ast \left( M / \langle m_1,\ldots,m_k \rangle \right) \, = \, 0.
\]  
A {\em $v$-basis of $M$} is a set of $v$-generators $m_1,\ldots,m_k$ such that $k$ is smallest possible. In this case, we call $k$ the {\em shift-dimension} (or {\em $v$-dimension}) of $M$ and denote it~$\dim_v (M).$
\end{definition}
Throughout the article, we are going to stick to {\em homogeneous} elements when studying generators and $v$-bases. The following lemma justifies that convention.
 \begin{lemma}
	If $M$ is a tame persistence module, then for all $v\in \Rnn^\param$ there exists a $v$-basis consisting of finitely many homogeneous generators. 
\end{lemma}
\begin{proof}
	This statement immediately follows from the fact that every finitely presented $\Rnn^r$-graded $\field[\Rnn^r]$-module can be minimally generated by homogeneous elements.
\end{proof}

The shift-dimension can be equally characterized as the smallest number of elements of $M$ such that $v\ast M$ is contained in the submodule generated by those elements. 
\begin{remark}
The $0$-dimension of $M$ is the minimal number of  generators of~$M.$ For a fixed vector $v,$ one may consider the family $\{\dim_{\tau v}(M) \}_{\tau\in\RR_{\geq 0}}$ of shift-dimensions of a multipersistence module $M.$ The non-increasing function $\RR_{\geq 0}\to \NN$ mapping $\tau \mapsto \dim_{\tau v}(M)$  is suitable as a feature map for machine learning~algorithms. 
\end{remark}
In general, $\Rnn^\param$-graded $\field[\Rnn^\param]$-modules are hard to handle. We will mostly restrict to the category of {\em finitely presented} $\Rnn^\param$-graded $\field[\Rnn^\param]$-modules and thereby to the category $\Tame\left( \Rnn^\param,\text{Vect}_{\field}\right)$ of {\em tame} functors.

\begin{example}[An indecomposable]\label{ex:indecomposable}
Consider the following commutative diagram of vector spaces and linear maps: 
\[	\begin{tikzpicture}[node distance=1.2cm, auto]
	\newcommand{\scalenumber}{.85}
	\node (00) {};
	\node (01) [above of=00] {};
	\node (02) [above of=01] {};
	\node (03) [above of=02] {};
	\node (04) [above of=03] {$\field$};
	\node (05) [above of=04] {$\field$};
	\node (06) [above of=05] {$\field$};
	\node (10) [right of=00] {};
	\node (20) [right of=10] {};
	\node (30) [right of=20] {};
	\node (40) [right of=30] {};
	\node (50) [right of=40] {};
	\node (60) [right of=50] {};
	\node (70) [right of=60] {};
	\node (80) [right of=70] {$\field$};
	\node (90) [right of=80] {$\field$};
	\node (100) [right of=90] {$\field$};
	\node (110) [right of=100] {$\field$};
	\node (11) [right of=01] {};
	\node (21) [right of=11] {};
	\node (31) [right of=21] {};
	\node (41) [right of=31] {};
	\node (51) [right of=41] {};
	\node (61) [right of=51] {$\field$};
	\node (71) [right of=61] {$\field$};
	\node (81) [right of=71] {$\field^2$};
	\node (91) [right of=81] {$\field$};
	\node (101) [right of=91] {$\field$};
	\node (111) [right of=101] {$\field$};
	\node (12) [right of=02] {};
	\node (22) [right of=12] {};
	\node (32) [right of=22] {};
	\node (42) [right of=32] {$\field$};
	\node (52) [right of=42] {$\field$};
	\node (62) [right of=52] {$\field^2$};
	\node (72) [right of=62] {$\field^2$};
	\node (82) [right of=72] {$\field^2$};
	\node (92) [right of=82] {$\field$};
	\node (102) [right of=92] {$\field$};
	\node (112) [right of=102] {$\field$};
	\node (13) [right of=03] {};
	\node (23) [right of=13] {};
	\node (33) [right of=23] {};
	\node (43) [right of=33] {$\field$};
	\node (53) [right of=43] {$0$};
	\node (63) [right of=53] {$\field$};
	\node (73) [right of=63] {};
	\node (83) [right of=73] {};
	\node (93) [right of=83] {};
	\node (103) [right of=93] {};
	\node (14) [right of=04] {$\field$};
	\node (24) [right of=14] {$\field^2$};
	\node (34) [right of=24] {$\field^2$};
	\node (44) [right of=34] {$\field^2$};
	\node (54) [right of=44] {$\field$};
	\node (64) [right of=54] {$\field$};
	\node (74) [right of=64] {};
	\node (84) [right of=74] {};
	\node (94) [right of=84] {};
	\node (104) [right of=94] {};
	\node (15) [right of=05] {$\field$};
	\node (25) [right of=15] {$\field^2$};
	\node (35) [right of=25] {$\field$};
	\node (45) [right of=35] {$\field$};
	\node (55) [right of=45] {};
	\node (65) [right of=55] {};
	\node (75) [right of=65] {};
	\node (85) [right of=75] {};
	\node (95) [right of=85] {};
	\node (105) [right of=95] {};
	\node (16) [right of=06] {$\field$};
	\node (26) [right of=16] {$\field$};
	\node (36) [right of=26] {};
	\node (46) [right of=36] {};
	\node (56) [right of=46] {};
	\node (66) [right of=56] {};
	\node (76) [right of=66] {};
	\node (86) [right of=76] {};
	\node (96) [right of=86] {};
	\node (106) [right of=96] {};
	\draw[->] (80) to node {} (90);
	\draw[->] (80) to node {\scalebox{\scalenumber}{\tiny ${\begin{pmatrix} 1\\ 1 \end{pmatrix}}$}} (81);
	\draw[->] (90) to node {} (100);
	\draw[->] (90) to node {} (91);
	\draw[->] (100) to node {} (101);
	\draw[->] (61) to node {} (71);
	\draw[->] (61) to node {\scalebox{\scalenumber}{\tiny ${\begin{pmatrix} 0 \\ 1 \end{pmatrix}}$}} (62);
	\draw[->] (71) to node {\scalebox{\scalenumber}{\tiny ${\begin{pmatrix} 0 \\ 1 \end{pmatrix}}$}} (81);
	\draw[->] (71) to node {\scalebox{\scalenumber}{\tiny ${\begin{pmatrix} 0 \\ 1 \end{pmatrix}}$}} (72);
	\draw[->] (81) to node {\scalebox{\scalenumber}{\tiny ${\begin{pmatrix} 0 \; 1 \end{pmatrix}}$}} (91);
	\draw[->] (81) to node {} (82);
	\draw[->] (91) to node {} (101);
	\draw[->] (91) to node {} (92);
	\draw[->] (101) to node {} (102);
	\draw[->] (42) to node {} (52);
	\draw[->] (42) to node {} (43);
	\draw[->] (52) to node {\scalebox{\scalenumber}{\tiny ${\begin{pmatrix} 1 \\ 0 \end{pmatrix}}$}} (62);
	\draw[->] (52) to node {} (53);
	\draw[->] (62) to node {} (72);
	\draw[->,right] (62) to node {\scalebox{\scalenumber}{\tiny ${\begin{pmatrix} 0 \; 1 \end{pmatrix}}$}} (63);
	\draw[->] (72) to node {} (82);
	\draw[->] (82) to node {\scalebox{\scalenumber}{\tiny ${\begin{pmatrix} 0 \; 1 \end{pmatrix}}$}} (92);
	\draw[->] (92) to node {} (102);
	\draw[->] (43) to node {} (53);
	\draw[->] (43) to node {\scalebox{\scalenumber}{\tiny ${\begin{pmatrix} 1 \\ 0 \end{pmatrix}}$}}  (44);
	\draw[->] (53) to node {} (63);
	\draw[->] (53) to node {} (54);
	\draw[->] (63) to node {} (64);
	\draw[->] (04) to node {} (05);
	\draw[->] (04) to node {} (14);
	\draw[->] (14) to node {\scalebox{\scalenumber}{\tiny ${\begin{pmatrix} 1 \\ 0 \end{pmatrix}}$}}  (24);
	\draw[->] (14) to node {} (15);
	\draw[->] (24) to node {} (34);
	\draw[->] (24) to node {} (25);
	\draw[->] (34) to node {} (44);
	\draw[->] (34) to node {\scalebox{\scalenumber}{\tiny ${\begin{pmatrix} 1 \; 0 \end{pmatrix}}$}}  (35);
	\draw[->] (44) to node {\scalebox{\scalenumber}{\tiny ${\begin{pmatrix} 0 \; 1 \end{pmatrix}}$}} (54);
	\draw[->] (44) to node {\scalebox{\scalenumber}{\tiny ${\begin{pmatrix} 1 \; 0 \end{pmatrix}}$}} (45);
	\draw[->] (54) to node {} (64);
	\draw[->] (05) to node {} (06);
	\draw[->] (05) to node {} (15);
	\draw[->] (15) to node {\scalebox{\scalenumber}{\tiny ${\begin{pmatrix} 1 \\ 0 \end{pmatrix}}$}} (25);
	\draw[->] (15) to node {} (16);
	\draw[->] (25) to node {\scalebox{\scalenumber}{\tiny ${\begin{pmatrix} 1 \; 0 \end{pmatrix}}$}} (35);
	\draw[->,right] (25) to node {\scalebox{\scalenumber}{\tiny ${\begin{pmatrix} 1 \; 1 \end{pmatrix}}$}} (26);
	\draw[->] (35) to node {} (45);
	\draw[->] (06) to node {} (16);
	\draw[->] (16) to node {} (26);
	\draw[->] (111) to node {} (112);
	\draw[->] (110) to node {} (111);
	\draw[->] (100) to node {} (110);
	\draw[->] (102) to node {} (112);
	\draw[->] (101) to node {} (111);
	\end{tikzpicture}
\]
This example is taken from~\cite{buchet2020every} and slightly modified. Maps between identical vector spaces are defined to be the identity map. We extend this representation to a two-parameter persistence module $M$ by reading the vertices of the quiver among $\{(i,j)\}_{i=0,\ldots,11,\, j=0,\ldots, 6}$ as depicted above,
by defining $M_{(i,j)}$ to be the trivial vector space for all remaining points of $\mathbb{N}^2,$ and to have identity maps $\id \colon (i,j)\to (i+\varepsilon_1,j+\varepsilon_2)$ for all $i,j\in\NN,$ $\varepsilon_1,\varepsilon_2\in[0,1).$ This persistence module is {\em indecomposable}, i.e., it is not the direct sum of any two non-trivial persistence modules. Let $v=(2,1).$ Then $\dim_0(M)=\beta_0(M)=5,$ $\dim_v(M)=2,$ and a $v$-basis is given by generators of the vector spaces at degrees $(0,4)$ and $(6,1).$  
\end{example}

At first glance, the notion of the shift-dimension might seem rather artificial. But, in fact, it arises in a natural way; namely as the hierarchical stabilization of $\beta_0.$
\begin{proposition}\label{prop stab nrank} 
	Let $M\in\Tame(\Rnn^\param,\Vect_{\field})$ be a tame persistence module.
	Denote by $\widehat{\rank_w}$ the hierarchical stabilization of the zeroth total multigraded Betti number w.r.t. the standard noise in the direction of $0\neq w\in\Rnn^\param$ and $v=\frac{w}{\|w\|}$ its normalization.~Then
	\[ \widehat{\rank_v}(M) (\delta) \,=\, \dim_{\delta v}(M). \]  
\end{proposition}

\begin{proof}
Denote by $C$ the standard contour in  the direction of $v.$
	As in \cite[Section~6.2]{noise}, we denote by
	$$ \mathcal{V}_{\delta}  \, \coloneqq\,  \left\{ N  \in \Tame (\Rnn^\param,\Vect_{\field}) \mid  \forall_{u\in \Rnn^\param:C(u,\delta)<\infty} \, \ker\left( N(u \leq u+ \delta \cdot v) \right) = N(u) \right\}$$
the noise system associated to the standard contour $C.$
	The system $\left\{ \mathcal{V}_{\delta} \right\}_{\delta \in \Rnn}$ is called the {\em standard noise in the direction of $v$}. 
Then, by~\cite[Theorem 8.3]{GC},
\[ \widehat{\rank_v}(M)(\delta)  \,=\,  \min\left\{ \rank(U) \mid U\in B(M,\delta) \right\}.\]
	Addition of the vector $\delta v$ in the definition of $ \mathcal{V}_{\delta}$ corresponds to the action of $\delta v$ on the module.
	Hence, in our case,
	\[ B(M,\delta) \,=\, \left\{ U\in \Tame (\Rnn^\param,\Vect_{\field}) \mid U\subseteq M, \, \delta v \ast (M/U)=0  \right\}, \]
and the claim follows.
\end{proof}
Note that $\widehat{\beta_0}_v(M)$ is a monotone, non-increasing step function. 

\begin{remark}
\cite[Theorem 8.3]{GC} is closely related to~\cite[Proposition~11.1]{noise}. The definition we use implies uniqueness properties of the shift, cf.~\cite[Proposition 8.1]{GC}. Furthermore, in~\cite{GC,noise}, noise systems are described for the non-negative {\em rational} numbers instead of the non-negative real numbers. The corresponding arguments we use generalize to $\Tame (\Rnn^\param,\Vect_{\field}).$
\end{remark}

\begin{remark}\label{remark:L}
As shown in \cite[Section 11]{GC}, computing the shift-dimension can be reduced to an NP-complete problem and hence is NP-hard itself:  for a certain class of persistence modules, the computation of the shift-dimension is at least as hard as the RANK-$3$ problem \cite{rank3}, an optimization problem in linear algebra. 
\end{remark}
 
We obtain a truncated version of the shift-dimension as stabilization of the zeroth multigraded Betti number by changing the contour to be truncated. For $C$ a contour and $\alpha \in \Rnn^\param,$ denote by $C_\alpha$ the {\em truncation} of $C$ at $\alpha.$ This is defined as 
\[
C_\alpha (x,\varepsilon) \, \coloneqq \, \begin{cases}
C(x,\varepsilon) & \text{if }  \alpha \not\leq C(x,\varepsilon),\\
\infty & \text{otherwise}.
\end{cases}
\]
 We call the corresponding noise system the {\em $\alpha$-truncated} noise system. 

\begin{proposition}\label{prop trunc stab nrank} 
	Under the assumptions of \Cref{prop stab nrank} with $\alpha$-truncated noise in the direction of a normalized vector $v,$ $\beta_0$ stabilizes as follows with respect to the $\alpha$-truncated standard contour in the direction of $v$:
\[ \widehat{\rank_{v,\alpha}}(M) (\delta) \,=\, \dim_{\delta v}(M/ \oplus_{\beta\geq \alpha}M_\beta). \]  
\end{proposition}

\begin{proof}
	The noise system associated to the truncated standard contour is given by
	$$\mathcal{V}_{\delta,\alpha}  \coloneqq  \left\{ N  \in \Tame ({\Rnn}^\param,\Vect_{\field}) \mid \forall_{u\in \Rnn^\param:C_\alpha(u,\delta)<\infty}  \ker\left( N(u \leq u+ \delta  v) \right) = N(u)  \right\}.$$
	We denote the balls corresponding to this noise system by  $B_{\alpha}$. Applying~\cite[Theorem 8.3]{GC}, we get
	\[ \widehat{\rank_{v,\alpha}}(M)(\delta)  \,=\,  \min\left\{ \rank(U) \mid U\in B_\alpha(M,\delta) \right\}.\]
	We have $M/U\in \mathcal{V}_{\delta,\alpha}$ iff for all homogeneous $m\in M/U$ we have $\delta v\ast m=0$ or $\alpha\leq \deg(\delta v \ast m).$ Hence,
	\[ B_\alpha(M,\delta) = \left\{ U\in \Tame (\Rnn^\param,\Vect_{\field}) \mid U\subseteq M, \, \delta v \ast ((M/\oplus_{\beta\geq \alpha}M_\beta)/[U])=0  \right\}, \]
	and thus the claim of the proposition follows.
\end{proof}
Therefore, computing the stabilized $\beta_0$ with respect to a truncated contour corresponds to truncating the module at degree $\alpha.$

\subsection{Properties of the shift-dimension}\label{section properties}
We now state some fundamental properties of the shift-dimension.  In particular, we investigate it regarding additivity. 
 
On the level of epimorphisms, the following holds true.
\begin{lemma}
	Let $\varphi\colon M \twoheadrightarrow N$ be an epimorphism of tame persistence modules. Then the following holds true: $\dim_v(N) \leq \dim_v(M)$ for all $v\in \Rnn^r.$
\end{lemma}
\begin{proof}
	Let $\{m_1,\ldots, m_{\dim_v(M)}\}$ be a $v$-basis of $M.$  Then $\varphi(m_1),\ldots,\varphi(m_{\dim_v(M)})$ \mbox{$v$-generate} $N$ and hence  $\dim_v(N)\leq \dim_v(M).$ 
\end{proof}
For monomorphisms, we do not get a corresponding inequality in the reverse direction, as the following counterexample demonstrates.
\begin{example}\label{ex:basisnonadditive}
 Let $v=(1,1),$ $M$ be the interval module generated in degrees $(0,2)$ and $(2,0),$ and $N$ the free module generated by a single element in degree ~$(0,0).$ Then $M  \xhookrightarrow{} N$ and $\dim_v(M)=2\nleq 1= \dim_v(N).$
\end{example}

\begin{lemma}\label{lemma:drop1}
	Let $v\in \Rnn^\param$ and $m\in M.$ Then $\dim_v (M/\langle m \rangle )$ is contained in \linebreak$\{\dim_v(M),\dim_v(M)-1\},$ i.e., taking the quotient by a submodule that is generated by a single element can drop the shift-dimension by one at most.
\end{lemma}
\begin{proof}
Assume there exist $m_1,\ldots,m_{\dim_v(M)-2}\in M/\langle m \rangle$ that are a $v$-basis of $M/\langle m \rangle.$ Then for any choice of representatives $\tilde{m}_i$ of $m_i,$ $m,\tilde{m}_1,\ldots,\tilde{m}_{\dim_v(M)-2}$ $v$-generate~$M.$ If follows that $\dim_v(M)\leq \dim_v(M)-1,$ which is a contradiction.
\end{proof}
In order to decide whether an element $m\in M$ can be extended to a $v$-basis $\{m,m_2,\ldots,m_{\dim_v(M)}\}$ of $M,$ one makes use of the following criterion.
\begin{lemma}
	Let $v\in \Rnn^\param$ and $m\in M.$ There exists a $v$-basis of $M$ containing $m$ if and only if $\dim_v (M/\langle m \rangle) =\dim_v(M)-1.$
\end{lemma}
\begin{proof}Let $\{m,m_2,\ldots,m_{\dim_v(M)}\}$ be a $v$-basis of $M.$ Then $[m_2],\ldots,[m_{\dim_v(M)}]$\linebreak  \mbox{$v$-generate} $M/ \langle m \rangle.$ Hence $\dim_v(M/\langle m \rangle) \leq \dim_v(M)-1$ and therefore \linebreak$\dim_v(M/\langle m \rangle) = \dim_v(M)-1$ by \Cref{lemma:drop1}. This proves the implication from left to right. Now let $\{m_1,\ldots,m_{\dim_v(M)-1}\}$ be a $v$-basis of $M/\langle m \rangle$ and choose representatives $\tilde{m}_1,\ldots,\tilde{m}_{\dim_v(M)-1}\in M$ for which $[\tilde{m}_i]=m_i,$ $i=1,\ldots,\dim_v(M)-1.$ Then $m,\tilde{m}_1,\ldots,\tilde{m}_{\dim_v(M)-1}$ $v$-generate $M$ and hence already are a $v$-basis of $M.$
\end{proof}

For short exact sequences, we prove the following two inequalities.
\begin{theorem}
	Let $0 \longrightarrow M \stackrel{\varphi}{\longrightarrow} L \stackrel{\psi}{\longrightarrow} N \to 0$ be a short exact sequence of persistence modules. Then for all $v,w\in \Rnn^r,$ the following two inequalities hold: 
	\begin{enumerate}[(i)]
		\item  $\dim_{v+w}(L)\leq \dim_v(M)+\dim_w(N),$ and
		\item $\dim_v(L) \leq \dim_v (N) + \beta_0(M).$
	\end{enumerate} 
\end{theorem} 
\begin{proof} 
	Choose a $w$-basis $\{n_1,\ldots,n_a\}$ of $N.$ 
	For every~$i,$ choose $l_i \in L$  for which $\psi(l_i)=n_i.$  Let $\{g_1,\ldots,g_{\beta_0(L)}\}$ generate $L.$ 
	Then for all $i$ there exist $\alpha_1,\ldots,\alpha_a\in \field[\Rnn^r]$ such that $w\ast  g_i - \sum_{j=1}^a \alpha_j l_j \in \ker(\psi)=\image (\varphi).$ 
	Choose a $v$-basis $\{m_1,\ldots,m_b\}$ of $M.$ 
	Thus, for all $i,$ there exist \mbox{$\gamma_1,\ldots,\gamma_{b} \in \field[\Rnn^r]$} 
	such that $v\ast ( w\ast  g_i - \sum_{j=1}^a \alpha_j l_j  ) = \varphi(\sum_{k=1}^b \gamma_k m_k),$ which is equivalent to   $(v+w)g_i =   \sum_{j=1}^a (v\ast \alpha_j) l_j + 
	\sum_{k=1}^b \gamma_k   \varphi(m_k).$ Consequently \mbox{$\dim_{v+w}(L) \leq b + a= \dim_v(M)+\dim_w(N),$} concluding the proof of (i).
	For the proof of statement (ii), let $F$ be the free module generated by a $v$-basis of~$N.$ Since $\psi \colon L \twoheadrightarrow N$ is surjective, we obtain a morphism $F \to L$ and hence a morphism from $M\oplus F$ to $L.$ Since $L/(M\oplus F)\cong  N/F,$ it follows that $\dim_v(L/(M\oplus F))=0.$ Therefore, $\dim_v(L)\leq \beta_0(M\oplus F) = \beta_0 (M) + \beta_0(F)=\beta_0(M) + \dim_v(N).$
\end{proof}
If $\dim_v(M) = \beta_0 (M)$ or $\beta_0(M)=1,$ it follows that $\dim_{v}(L)\leq \dim_v(M)+\dim_v(N).$ 

The shift-dimension is sub-additive, i.e., $\dim_v(M\oplus N)\leq \dim_v (M) + \dim_v(N).$
In general, it is not additive, as the following counterexample for L-shapes demonstrates.
\begin{example}[$\param=2$]\label{counterexampleadd}
		Consider the interval modules $M=\field((1,0),\bullet)  / \field((1,2),\bullet)$ and $N = \field((0,1),\bullet) /\field((2,1),\bullet).$ Then $(1,1)\ast (M\oplus N)$ is  contained in the submodule of $M\oplus N$  that is generated by the element $(1,1)\in (M\oplus N)_{(1,1)}$ in degree $(1,1).$
		It follows that $\dim_{(1,1)}(M\oplus N)=1 \neq 1+1= \dim_{(1,1)}(M)+\dim_{(1,1)}(N).$
\end{example}
We now present some cases for which additivity of the shift-dimension does hold~true.
\begin{lemma}
	If $F$ and $G$ are two free persistence modules of finite rank, then $\dim_v(F\oplus G)=\dim_v (F)+\dim_v(G)$ for all $v\in \Rnn^r.$ 
\end{lemma}
\begin{proof}
	The statement follows from the fact that in sufficiently large degree, the elements of the $v$-basis need to generate $\field^{\beta_0(F)+\beta_0(G)}.$
\end{proof}

\begin{proposition}
	Let $M$ be an interval module generated by one element and $F$ a free persistence module of rank one. Then $\dim_v(M\oplus F)=\dim_v(M)+1$ for all $v\in \Rnn^r.$
\end{proposition}
\begin{proof}
	For $v\in \Rnn^r$ s.t.\ $\dim_v(M)=0,$ the statement is clear. Hence let $v\in \Rnn^r$ be such that $\dim_v(M)=1.$ Then  $\dim_v(M\oplus F)\leq \dim_v(M)+\dim_v(F)=1+1=2$ and $\dim_v(M\oplus F) \geq \max \{\dim_v(M),\dim_v(F)\}=1.$ 
	Assume $U=\langle (m,f)\rangle $ for some $(m,f)\in M\oplus F.$ Then  $ v\ast ((M\oplus F) / U)  \not\cong 0,$ as we argue now. If $f=0,$ then for all $g\in F\setminus \{0\},$ the element $v\ast (0,g)\notin U.$ If $f\neq 0,$ then there exists $m\in M\setminus \{0\}$ such that $v\ast (m,0)\notin U,$ since by assumption there exists $m\in M$ for which $v\ast m\neq 0.$ Hence $\dim_v(M\oplus F)>1,$ proving the claim.
\end{proof}

\begin{proposition}
	Let $\dim_v(M)\leq 1$ and $F$ a free persistence module of rank $k.$ Then $\dim_v (M\oplus F)=\dim_v(M)+\dim_v(F).$
\end{proposition}
\begin{proof}
For $\dim_v(M)=0,$ the statement is clear. Now assume that $\dim_v(M)=1.$ Assume there exists a $v$-basis $\{ (m_1,f_1),\ldots,(m_k,f_k)\}$ $M$ of cardinality $k.$ Since $\dim_v(F)=k$ and $F$ is free, $f_1,\ldots,f_k$ need to be linearly independent. By assumption on the shift-dimension of $M,$ there exists $m\in M$ s.t.\ $v\ast m\neq 0.$ Because of the linear independence of $f_1,\ldots,f_k,$ the element $(v\ast m,0)$ is not contained in the submodule generated by $(m_1,f_1),\ldots,(m_k,f_k),$ in contradiction to the assumption.
\end{proof}

\subsection{A quantitative study of non-additivity}\label{section:locus}
In this subsection, we investigate non-additivity of the shift-dimension in greater detail and give a  measure for it. 
For a finite family of persistence modules $\{ M_i\}_{i \in I}$ and a fixed, possibly learned $v\in{\Rnn}^r,$ we define the {\em locus of non-additivity} of the shift-dimension as 
$$ \Loc_v \left( \left\{ M_i\right\}_{i \in I} \right) \, \coloneqq \, \left\{ \tau \in \Rnn \mid \dim_{\tau v} \left( \oplus_{i\in I} M_i \right) \neq \sum_{i \in I} \dim_{\tau v } \left( M_i  \right)\right\}.$$ 
For a quantitative study,  we associate to $\{M_i\}_{i\in I}$ the $L^p$-distance between the functions $\sum_{i\in I} \dim_{\bullet }(M_i)$ and $\dim_{\bullet }\left(\bigoplus_{i\in I}M_i\right)$ of $\tau,$ i.e., for $1\leq p<\infty,$ we study
$$\err_{v,p}\left( \left\{M_i\right\}_{i \in I}\right) \, \coloneqq \, \left( \int_{0}^\infty \left(\left(\sum_{i\in I} \dim_{\tau v}(M_i) \right) - \dim_{\tau v}\left(\bigoplus_{i\in I}M_i\right)\right)^p \diff \tau \right)^{1/p}.$$

If this expression  is sufficiently small, the sum of the shift-dimensions of the $M_i$ yields a good approximation of the shift-dimension  of their direct sum $\oplus M_i.$
Revisiting and generalizing \Cref{ex:basisnonadditive}, we now undertake more quantitative investigations.
\begin{example}\label{exnonaddloc}
	Let $v\in\mathbb{R}^r_{>0}.$ Let $M_1,M_2\neq0$ be interval modules generated by single elements $\gen_1$ and~$\gen_2,$ resp., and quotiented out by single elements $\widetilde{\gen}_1$ and $\widetilde{\gen}_2,$ resp. Assume  further that 
$\text{lcm}(\deg({\gen}_1),\deg({\gen}_2))<\deg(\widetilde{\gen}_2)\leq\deg(v\ast {\gen}_1)$ and $\text{lcm}(\deg({\gen}_1),\deg({\gen}_2))<\deg(\widetilde{\gen}_1)\leq\deg(v\ast {\gen}_2)$ 
as well as
$\deg(\widetilde{\gen}_1)\nleq\deg(v\ast{\gen}_1)$ and $\deg(\widetilde{\gen}_2)\nleq\deg(v\ast{\gen}_2).$
Then, $\text{dim}_v(M_1\oplus M_2)=1\neq 2=\text{dim}_v(M_1)+\text{dim}_v(M_2),$ since the element $(1,1)\in(M_1\oplus M_2)_{\text{lcm}(\deg(\gen_1),\deg(\gen_2))}$ divides both $v\ast(\gen_1,0)$ and $v\ast(0,\gen_2).$
	Now, for all $\tau\in\Loc_v(\{M_i\}_{i \in I})$ we get $\text{dim}_{\tau v}(M_1\oplus M_2)=1\neq 2=\text{dim}_{\tau v}(M_1)+\text{dim}_{\tau v}(M_2).$ The length of the interval $\Loc_v(\{M_i\}_{i \in I})$ yields $\err_{v,1}(\{M_1,M_2\}).$
\end{example}
Let $M_1,M_2$ be interval modules as in the example above. In this case, the union of their underlying intervals is an interval.
Denote by $M$ the corresponding interval module. Then, whenever $\text{dim}_{\tau v}(M_1\oplus M_2)=1,$ we have $\text{dim}_{\tau v}(M_1\oplus M_2)=\text{dim}_{\tau v}(M).$ 
 
\begin{figure}
\hspace*{-12mm}
\begin{minipage}[t]{0.08\textwidth}
\end{minipage}
\begin{minipage}[t]{0.47\textwidth}
\scalebox{.9}{
		\begin{tikzpicture}[scale=0.6, transform shape]
			\tikzstyle{grid lines}=[lightgray,line width=0]
			\foreach \r in {0,1,..., 9}
			\draw (\r,0) node[inner sep=.5pt,below=9pt,rectangle,fill=white] {\fontsize{13}{13}$\r$};
			\foreach \r in {0,1, 2,...,9}
			\draw (0,\r) node[inner sep=1pt,left=9pt,rectangle,fill=white] {\fontsize{13}{13}$\r$};
			\tikzstyle{every node}=[circle, draw, fill=niceblue, inner sep=1pt, minimum width=9pt] node{};
			\draw[loosely dotted] (0,4) -- (3,1);
			\foreach \r in {1,...,9}
			\foreach \s in {1,...,9}
			\node[fill=none,draw=none] at (\r,\s) {.};
			\foreach \r in {1,...,9} 
			\draw[black] (\r,-0.05) -- (\r,0.05);
			\foreach \r in {1,...,9} 
			\draw[black] (-0.05,\r) -- (0.05,\r);
				\fill[\modulegray] (0.7,3.3) rectangle (5.3,7.8);
				\fill[\shiftgray] (4.7,7.3) rectangle (5.3,7.8);
 				\draw[black] (0.7,3.3)  -- (0.7,7.8) ;
 				\draw[black] (0.7,3.3)  -- (5.3,3.3) ;
 				\draw[black] (5.3,7.8)  -- (5.3,3.3) ;
 				\draw[black] (5.3,7.8)  -- (0.7,7.8) ;
				%
				\fill[\modulegray] (2,2) rectangle (6.5,6.5);
				\fill[\shiftgray] (6,6) rectangle (6.5,6.5);
 				\draw[black] (2,2)  -- (2,6.5) ;
				\draw[black] (2,2)  -- (6.5,2) ;
				\draw[black] (6.5,6.5)  -- (2,6.5) ;
				\draw[black] (6.5,6.5)  -- (6.5,2) ;
				%
				\fill[\modulegray] (0,4) rectangle (4.3,9.5);
				\fill[\shiftgray] (4,8) rectangle (4.3,9.5);
 				\draw[black] (0,4)  -- (4.3,4) ;
 				\draw[black] (4.3,4)  -- (4.3,9.5) ;
				%
				\fill[\modulegray] (3,1) rectangle (9.5,5.5);
				\fill[\shiftgray] (7,5) rectangle (9.5,5.5);
 				\draw[black] (3,1)  -- (3,5.5) ;
 				\draw[black] (3,5.5)  -- (9.5,5.5) ;
 				\draw[black] (3,1)  -- (9.5,1) ;
				\draw[black,->] (0,0)  -- (0,9.5) ;
				\draw[black,->]  (0,0)  -- (9.5,0) ;
				\draw (0,4) node[fill=\modulegray]{};
				\draw (3,1) node[fill=\modulegray]{};
				\draw (0.7,3.3) node[fill=\modulegray]{};
				\draw (2,2) node[fill=\modulegray]{};
				\draw (4,8) node[fill=\shiftgray]{};
				\draw (4.7,7.3) node[fill=\shiftgray]{};
				\draw (6,6) node[fill=\shiftgray]{};
				\draw (7,5) node[fill=\shiftgray]{};
				\draw[dashed,\shiftgray] (4,8) -- (4,5);
				\draw[dashed,\shiftgray] (4.7,7.3) -- (4,5);
				\draw[dashed,\shiftgray] (6,6) -- (4,5);
				\draw[dashed,\shiftgray] (7,5) -- (4,5);
				\draw (4,5) node[fill=niceblue!90]{};
				\draw (4,4.5) node[draw=none,fill=none] {\fontsize{13}{13}\textcolor{niceblue!90}{$(1,1,1,1)$}};
		\end{tikzpicture}  }
\end{minipage}\hspace*{-8mm}
\begin{minipage}[t]{0.43\textwidth}
\scalebox{.9}{
		\newcommand{\scale}{1.8}
		\newcommand{\widthadjust}{1.6}
		\begin{tikzpicture}[scale=0.6, transform shape]
			\tikzstyle{grid lines}=[lightgray,line width=0]
			\foreach \r in {0,1,...,6}
			\draw (\scale*\r,0) node[inner sep=.5pt,below=11pt,rectangle,fill=white] {\fontsize{13}{13}$\r$};
			\foreach \r in {0,1, 2,...,5}
			\draw (0,\scale*\r) node[inner sep=1pt,left=11pt,rectangle,fill=white] {\fontsize{13}{13}$\r$};
			\tikzstyle{every node}=[circle, draw, fill=niceblue, inner sep=1pt, minimum width=8pt] node{};
			\foreach \r in {1,...,6}
			\foreach \s in {1,...,5}
			\node[fill=none,draw=none] at (\scale*\r,\scale*\s) {.};
				\draw[black,->] (0,0)  -- (\scale*0,9.5) ;
				\draw[black,->]  (0,0)  -- (9.5,0) ;
				\foreach \r in {1,...,6} 
				\draw[black] (\scale*\r,-\scale*0.05) -- (\scale*\r,\scale*0.05);
				\foreach \r in {1,...,5} 
				\draw[black] (-\scale*0.05,\scale*\r) -- (\scale*0.05,\scale*\r);
				\draw[niceblue!90,line width=\scale*\widthadjust*\goodwith] (\scale*0,\scale*3.92) -- (\scale*3,\scale*3.92);
				\draw[niceblue!90,line width=\scale*\widthadjust*\goodwith] (\scale*3,\scale*2) -- (\scale*3.6,\scale*2);
				\draw[niceblue!90,line width=\scale*\widthadjust*\goodwith] (\scale*3.6,\scale*.92) -- (\scale*4.5,\scale*.92);
				\draw[niceblue!90,line width=\scale*\widthadjust*\goodwith] (\scale*4.3,\scale*.92) -- (\scale*4.5,\scale*.92);
				\draw[niceblue!90,line width=\scale*\widthadjust*\goodwith] (\scale*4.5,\scale*-.08) -- (\scale*6.5,\scale*-.08);
				\draw[\goodred,line width=\scale*\widthadjust*\goodwith] (\scale*0,\scale*4.08) -- (\scale*4.3,\scale*4.08);
				\draw[\goodred,line width=\scale*\widthadjust*\goodwith] (\scale*4.3,\scale*1.08) -- (\scale*4.5,\scale*1.08);
				\draw[\goodred,line width=\scale*\widthadjust*\goodwith] (\scale*4.5,\scale*0.08) -- (\scale*6.5,\scale*0.08);
		\end{tikzpicture}  }
\end{minipage}
\caption{Illustration of a persistence module as in~\Cref{ex:rectanglenonadditivity} for~$|I|=2$. \emph{Left:} Let $v=(4,4).$ All shifted minimal generators have pairwise incomparable degrees but can be generated by the all-one vector in the degree of the greatest common divisor of their degrees. Hence, $\dim_v(\oplus_i(M_i))=1$. \emph{Right:} The functions $\sum_i\dim_{\tau v}(M_i)$ (in lavender) and  $\dim_{\tau v}(\oplus_i(M_i))$ (in blue) for $v=(1,1).$  
	We have $\Loc_v ( \{ M_i\}_{i \in I}=[3,4.3)$ and $\err_{v,p}=(0.6\cdot 2^p+0.7\cdot 3^p)^{1/p}.$}
\label{fig:rectanglenonadditivity}
\end{figure}
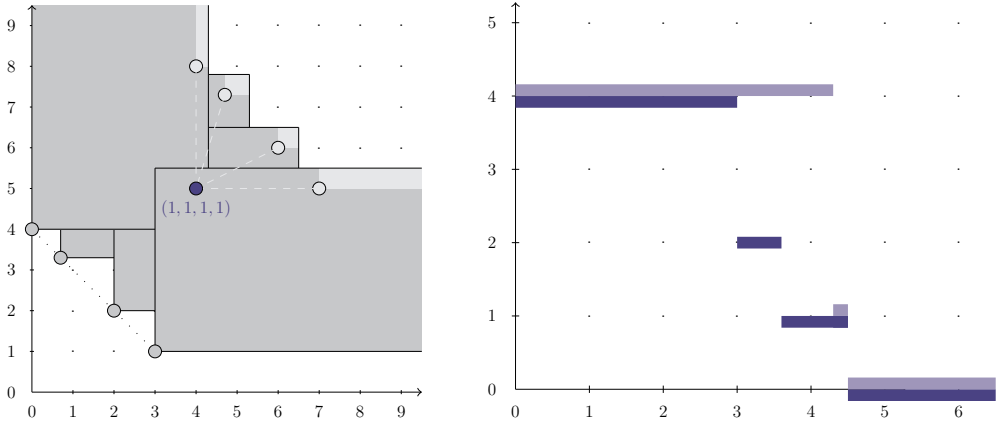
The following example demonstrates that the difference between $\dim_v(\oplus_{i\in I}M_i)$ and $\sum_{i\in I}\dim_v(M_i)$ can be arbitrarily large. In fact, the shift-dimension of a module can be~$1$ while the Hilbert function and the rank invariant attain arbitrarily large values.
\begin{example}\label{ex:rectanglenonadditivity}
Let $M_1$ and $M_2$ be as in~\Cref{exnonaddloc}. Define a  finite collection of interval modules $\{M_i\}_{i\in I}$ such that each $M_i$ has one generator $\gen_i$ and gets quotiented out by two elements $\widetilde{\gen}_i,\, \widehat{\gen}_i$ that are of the following form. The  $\{\gen_i\}_{i\in I}$ are assumed to have pairwise distinct degrees which lie on the straight line between $\deg(\gen_1)$ and $\deg(\gen_2).$ The degrees of $\widetilde{\gen_i}$ and $\widehat{\gen_i}$ are chosen such that $\deg(\widetilde{\gen}_i),\deg(\hat{\gen}_i)\nleq \deg(v\ast \gen_i)$ and for all $j\neq i$ either $\deg(\widetilde{\gen}_i)\leq \deg(v\ast\gen_j)$ or $\deg(\widehat{\gen}_i) \leq \deg(v\ast\gen_j).$ Then, again, the all-one vector in degree $\operatorname{lcm}(\deg(\gen_1),\deg(\gen_2))$ is a divisor of each $v\ast (0, \ldots,\gen_i,\ldots,0). $ 
Hence,  $\dim_v(\bigoplus_{\{1,2\}\cup I} M_i)=1\neq 2+|I|=\sum_{\{1,2\}\cup I}\text{dim}_v(M_i).$ 
See~\Cref{fig:rectanglenonadditivity} for an illustration of one concrete example.
\end{example}

Note that $\Loc_v (\{M_i\}_{i \in I})$ is in general not connected. Counterexamples can be constructed by taking direct sums of certain modules as in the examples above.

We give a further, more involved example for which additivity does not hold true.

\begin{figure}[h]
	\centering
	\begin{minipage}[t]{0.47\textwidth}
		\scalebox{.65}{
			\begin{tikzpicture}[scale=0.6, transform shape]
			\tikzstyle{grid lines}=[lightgray,line width=0]
			\foreach \r in {0,1,..., 15}
			\draw (\r,0) node[inner sep=.5pt,below=9pt,rectangle,fill=white] {\fontsize{13}{13}$\r$};
			\foreach \r in {0,1, 2,...,13}
			\draw (0,\r) node[inner sep=1pt,left=9pt,rectangle,fill=white] {\fontsize{13}{13}$\r$};
			\tikzstyle{every node}=[circle, draw, fill=niceblue, inner sep=1pt, minimum width=12pt] node{};
			\foreach \r in {1,...,15}
			\foreach \s in {1,...,13}
			\node[fill=none,draw=none] at (\r,\s) {.};
			\fill[\modulegray] (0,8) rectangle (15.5,13.5);
			\fill[\modulegray] (6,4) rectangle (15.5,13.5);
			\fill[\modulegray] (4,6) rectangle (15.5,13.5);
			\fill[\modulegray] (8,2) rectangle (15.5,13.5);
			\fill[\modulegray] (11,0) rectangle (15.5,13.5);
			\draw[black,->] (0,0)  -- (0,13.5) ;
			\draw[black,->]  (0,0)  -- (15.5,0) ;
			\foreach \r in {1,...,15} 
			\draw[black] (\r,-0.05) -- (\r,0.05);
			\foreach \r in {1,...,13} 
			\draw[black] (-0.05,\r) -- (0.05,\r);
			\fill[\shiftgray] (4,12) rectangle (15.5,13.5);
			\fill[\shiftgray] (10,8) rectangle (15.5,13.5);
			\fill[\shiftgray] (8,10) rectangle (15.5,13.5);
			\fill[\shiftgray] (12,6) rectangle (15.5,13.5);
			\fill[\shiftgray] (15,4) rectangle (15.5,13.5);
			\draw[black] (0,8)  -- (4,8) ;
			\draw[black] (4,8)  -- (4,6) ;
			\draw[black] (4,6)  -- (6,6) ;
			\draw[black] (6,6)  -- (6,4) ;
			\draw[black] (6,4)  -- (8,4) ;
			\draw[black] (8,4)  -- (8,2) ;
			\draw[black] (8,2)  -- (11,2) ;
			\draw[black] (11,2)  -- (11,0) ;				
			\draw[\shiftgray] (4,12)  -- (4,13.5) ;
			\draw[\shiftgray] (4,12)  -- (8,12) ;
			\draw[\shiftgray] (8,10)  -- (8,12) ;
			\draw[\shiftgray] (8,10)  -- (10,10) ;
			\draw[\shiftgray] (10,8)  -- (10,10) ;
			\draw[\shiftgray] (10,8)  -- (12,8) ;
			\draw[\shiftgray] (12,8)  -- (12,6) ;
			\draw[\shiftgray] (12,6)  -- (15,6) ;
			\draw[\shiftgray] (15,4)  -- (15,6) ;
			\draw[\shiftgray] (15,4)  -- (15.5,4);
			\draw (4,12) node[fill=\shiftgray]{};
			\draw (10,8) node[fill=\shiftgray]{};
			\draw (8,10) node[fill=\shiftgray]{};
			\draw (12,6) node[fill=\shiftgray]{};
			\draw (15,4) node[fill=\shiftgray]{};
			\draw (0,8) node[fill=\modulegray]{};
			\draw (6,4) node[fill=\modulegray]{};
			\draw (4,6) node[fill=\modulegray]{};
			\draw (8,2) node[fill=\modulegray]{};
			\draw (11,0) node[fill=\modulegray]{};
			\draw[niceblue!90,line width=\goodwith] (8,4)  -- (8,11.5) ;
			\draw[niceblue!90,line width=\goodwith] (8,4)  -- (15.5,4) ;
			\draw[niceblue!90,line width=\goodwith] (3.5,11.5)  -- (3.5,13.5) ;
			\draw[niceblue!90,line width=\goodwith] (3.5,11.5)  -- (8,11.5) ;
			\draw (8,4) node[fill=niceblue!90]{};
			\draw (3.5,11.5) node[fill=niceblue!90]{};
			\draw[\goodred,line width=\goodwith] (4,6)  -- (4,13.5) ;
			\draw[\goodred,line width=\goodwith] (4,6)  -- (13,6) ;
			\draw[\goodred,line width=\goodwith] (13,4)  -- (13,6) ;
			\draw[\goodred,line width=0.7*\goodwith] (13,3.93)  -- (15.5,3.93) ;
			\draw (4,6) node[fill=\goodred]{};
			\draw (13,4) node[fill=\goodred]{};
			\end{tikzpicture}  }
	\end{minipage}
	\begin{minipage}[t]{0.05\textwidth}
	\end{minipage}
	\begin{minipage}[t]{0.45\textwidth}
		\scalebox{.65}{
			\begin{tikzpicture}[scale=0.6, transform shape]
			\tikzstyle{grid lines}=[lightgray,line width=0]
			\foreach \r in {0,1,..., 15}
			\draw (\r,0) node[inner sep=.5pt,below=9pt,rectangle,fill=white] {\fontsize{13}{13}$\r$};
			\foreach \r in {0,1, 2,...,13}
			\draw (0,\r) node[inner sep=1pt,left=9pt,rectangle,fill=white] {\fontsize{13}{13}$\r$};
			\tikzstyle{every node}=[circle, draw, fill=blue, inner sep=1pt, minimum width=12pt] node{};
			\foreach \r in {1,...,15}
			\foreach \s in {1,...,13}
			\node[fill=none,draw=none] at (\r,\s) {.};
			\foreach \r in {1,...,15} 
			\draw[black] (\r,-0.05) -- (\r,0.05);
			\foreach \r in {1,...,13} 
			\draw[black] (-0.05,\r) -- (0.05,\r);
			\fill[\modulegray] (0,8) rectangle (15.5,13.5);
			\fill[\modulegray] (6,4) rectangle (15.5,13.5);
			\fill[\modulegray] (4,6) rectangle (15.5,13.5);
			\fill[\modulegray] (8,2) rectangle (15.5,13.5);
			\fill[\modulegray] (11,0) rectangle (15.5,13.5);
			\fill[\shiftgray] (4,12) rectangle (15.5,13.5);
			\fill[\shiftgray] (10,8) rectangle (15.5,13.5);
			\fill[\shiftgray] (8,10) rectangle (15.5,13.5);
			\fill[\shiftgray] (12,6) rectangle (15.5,13.5);
			\fill[\shiftgray] (15,4) rectangle (15.5,13.5);
			\draw[black] (0,8)  -- (4,8) ;
			\draw[black] (4,8)  -- (4,6) ;
			\draw[black] (4,6)  -- (6,6) ;
			\draw[black] (6,6)  -- (6,4) ;
			\draw[black] (6,4)  -- (8,4) ;
			\draw[black] (8,4)  -- (8,2) ;
			\draw[black] (8,2)  -- (11,2) ;
			\draw[black] (11,2)  -- (11,0) ;		
			\fill[niceblue!90] (0,8) rectangle (4,12);
			\fill[white!20!niceblue] (3.98,6) rectangle (4.02,8);
			\fill[\goodred] (8,2) rectangle (15,4);
			\fill[\goodred] (11,0) rectangle (15,4);
			\fill[\goodred] (6,3.98) rectangle (8,4.02);
			\draw[black,->] (0,0)  -- (0,13.5) ;
			\draw[black,->]  (0,0)  -- (15.5,0) ;
			\draw (6,4) node[fill=\goodred]{};
			\draw (4,6) node[shading=true,left color=white!20!niceblue,right color=\goodred]{};
			\draw[\shiftgray] (4,12)  -- (4,13.5) ;
			\draw[\shiftgray] (4,12)  -- (8,12) ;
			\draw[\shiftgray] (8,10)  -- (8,12) ;
			\draw[\shiftgray] (8,10)  -- (10,10) ;
			\draw[\shiftgray] (10,8)  -- (10,10) ;
			\draw[\shiftgray] (10,8)  -- (12,8) ;
			\draw[\shiftgray] (12,8)  -- (12,6) ;
			\draw[\shiftgray] (12,6)  -- (15,6) ;
			\draw[\shiftgray] (15,4)  -- (15,6) ;
			\draw[\shiftgray] (15,4)  -- (15.5,4);
			\draw (4,12) node[fill=\shiftgray]{};
			\draw (10,8) node[fill=\shiftgray]{};
			\draw (8,10) node[fill=\shiftgray]{};
			\draw (12,6) node[fill=\shiftgray]{};
			\draw (15,4) node[fill=\goodred]{};
			\draw[niceblue!90,line width=\goodwith] (4,6)  -- (4,12) ;
			\draw[\goodred,line width=\goodwith] (6,4)  -- (8,4) ;
			\draw[\goodred,line width=\goodwith] (11,0)  -- (11,2) ;
			\draw[\goodred,line width=\goodwith] (8,2)  -- (11,2) ;
			\draw (0,8) node[fill=niceblue]{};
			\draw (8,4) node[fill=niceblue!90]{};
			\draw (8,2) node[fill=\goodred]{};
			\draw (11,0) node[fill=\goodred]{};
			\draw (6,4) node[fill=\goodred]{};
			\draw (4,12) node[fill=niceblue!90]{};
			\draw (4,6) node[shading=true,left color=white!20!niceblue,right color=\goodred]{};
			\draw (8,4) node[shading=true,left color=white!20!niceblue,right color=\goodred]{};
			\end{tikzpicture}  }
	\end{minipage}
	\caption{In gray: the interval module $M$ of \Cref{example intervalmod}. In lighter gray: \mbox{$(4,4)\ast M.$} {\em Left:} 
		In blue and lavender: the submodules generated by the bases $\basis_1$ and~$\basis_2,$ resp. {\em Right:}  Regions in which the basis elements can be exchanged, fixing  $(8,4)$ and $(4,6),$ resp., following~\Cref{lemma:basis_exchange}. Using this lemma again, both bases can be exchanged with $\basis_3\coloneqq \{(4,6),(6,4)\}.$ Those replacements do not preserve the indispensability relations, since $(4,6)$ is indispensable only for ${(0,8)}$ and $(6,4)$ is indispensable only for $(11,0).$}
	\label{figure matroid}
\end{figure}
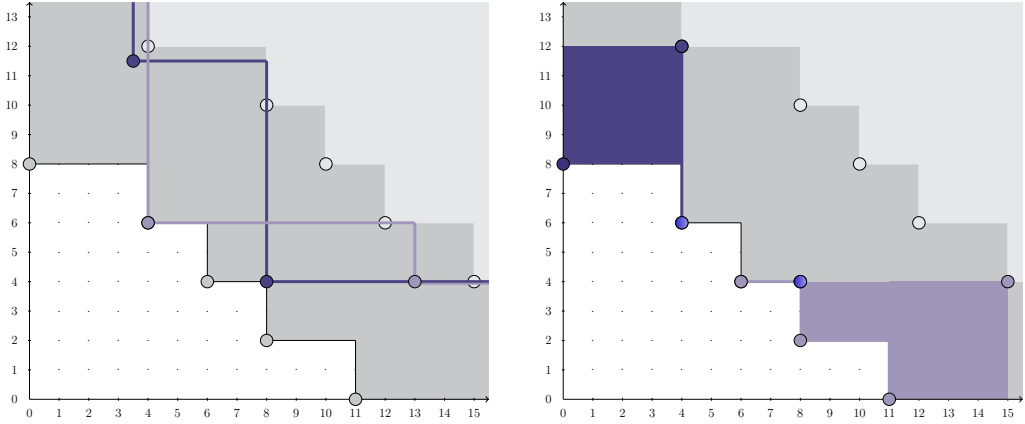

\begin{example}[\Cref{ex:indecomposable} revisited]
Let $M_1$ be the persistence module from \Cref{ex:indecomposable} and $M_2$ be an interval module generated at degrees $(0,5)$ and $(5,1.5)$ and quotiented out at degrees $(2,6)$  and $(9,1.5).$ 
Let $v=(2,1).$ 
Then, 
$$ \dim_{\tau v}(M_1) \,=\,\begin{cases}
5 & \text{if }0\leq\tau<1,\\
2 & \text{if }1\leq\tau<2,\\
0 & \text{if } 2 \leq \tau ,
\end{cases} \qquad \dim_{\tau v}(M_2) \,=\, \begin{cases}
2 & \text{if }0\leq\tau<1, \\
1 & \text{if }1\leq\tau<2,\\
0 & \text{if }2\leq\tau,
\end{cases}$$
and
$$\dim_{\tau v}(M_1\oplus M_2) \,=\, \begin{cases}
7 & \text{if }0\leq\tau<1,\\
3 & \text{if }1\leq\tau<1.5,\\
2 & \text{if }1.5\leq\tau<2,\\
0 & \text{if }2\leq\tau.
\end{cases}$$
We get $\Loc_v (\{M_1,M_2\})=[1.5,2)$ and $\err_{v,p}(\{M_1,M_2\})=0.5^{1/p}.$
\end{example}
Here, the non-triviality of the loci and $L^p$-errors of non-additivity is due to the existence of $v$-basis elements such that a suitable multiple of them yields a non-zero shifted generator in one direct summand and zero on the other direct summands.

\subsection{Basis exchange properties}
Unlike bases of vector spaces in linear algebra, \mbox{$v$-bases} do {\em not} give rise to a matroid if $\param>1.$ The following example demonstrates that even for monomial ideals, neither the well-desired basis exchange property holds, nor that for every set ot \mbox{$v$-generators}  there exists a subset which is a $v$-basis of $M.$ 
\begin{example}\label{example intervalmod}
Let $M$ be the interval module generated at degrees $(0,8),$ $(4,6),$ $(6,4),$ $(8,2),$ and $(11,0).$ For $v=(4,4),$ $\dim_v(M)$=2. 
Both $\basis_1\coloneqq\{(3.5,11.5),(8,4)\}$ and $\basis_2\coloneqq\{(4,12),(8,4)\}$ are $v$-bases of $M.$ The basis elements cannot be exchanged, as $(3.5,11.5),(4,6)\not\leq (4,4)+(11,0)$ and $(4,12),(13,4)\not\leq (4,4)+(8,2)$ (cf. \Cref{figure matroid}).  Taking the shifts of all $5$ minimal generators by $v$  yields a set of $v$-generators of~$M.$ However, none of its proper subsets is a set of \mbox{$v$-generators} of $M.$
\end{example}

We now investigate how to modify elements of a given $v$-basis by multiplication with suitable elements of the monoid. Let $m\in M$ and $\basis $ be a $v$-basis of $M.$ If for an  element $b$ of $\basis$ we have that $v\ast m\notin \langle \basis \setminus\{b\}\rangle ,$ we call $b$ {\em indispensable} for $m.$ Otherwise, we call $b$ {\em dispensable} for $m.$ 

\begin{lemma}\label{lemma:basis_exchange}
Let $\basis$ be a $v$-basis of a persistence module $M.$ Let $b\in\basis.$ Let $\{\gen_i\}_{i\in I}$ be a minimal set of homogeneous generators of~$M.$ Then the following holds.
\begin{enumerate}[(i)]
\item If there exists $\tilde{b}\in M$ and $r\in \field[\Rnn^\param]$ such that $r\tilde{b}=b,$ then $(\basis\setminus\{b\})\cup\{\tilde{b}\}$ is a $v$-basis of $M.$
\item Let $J\subseteq I$ be the set of all $j\in I$ s.t.\ $b$ is indispensable for $\gen_j.$ Let $w\in \Rnn^\param$ s.t.\  $\deg(w\ast b)\leq \gcd(\{\deg(v\ast\gen_j)\mid j\in J\}).$ Then  $(\basis\setminus\{b\})\cup\{w\ast  b\}$ is a $v$-basis~of~$M.$
\end{enumerate}
\end{lemma}

\begin{proof}
Let $\tilde{b}$ be as in the assumption of $(i).$
Then $(\basis\setminus\{b\})\cup\{\tilde{b}\}$ is a set of \mbox{$v$-generators} as it generates $\langle \basis \rangle $ by construction.
Since it has the same cardinality as $\basis,$ it is indeed a $v$-basis.
For the proof of the second statement, let $w$ be as desired. We prove that for each $\gen\in\{\gen_i\}_{i\in I}$ we have $v\ast\gen\in \langle (\basis\setminus\{b\})\cup\{w\ast  b\}\rangle .$ If $b$ is dispensable for~$\gen,$ then $v\ast\gen\in \langle \basis\setminus\{b\} \rangle\subseteq \langle (\basis\setminus\{b\})\cup\{w\ast  b\}\rangle .$ 
If $b$ is indispensable for~$\gen,$ then $v\ast\gen$ can be expressed as linear combination of elements of $(\basis\setminus\{b\})\cup\{w\ast  b\},$ since  $\deg(w\ast b)\leq\deg(v\ast\gen).$
\end{proof}

Given $b\in \basis,$ there is---up to constants in the field $\field$---a unique maximal multiple of~$b$ by which one can replace it, but there might be several incomparable such choices for divisors of $b.$
If one replaces a $v$-basis element by a multiple  or by a divisor  of it as described in the lemma above, it is important to note that the indispensability relations of the other $v$-basis elements with the minimal generators of $M$ might change.
Indispensability therefore encodes the combinatorial complexity of $v$-bases. 

\begin{corollary}\label{corollary:basis_exchange}
Let $M$ be an interval module, $\mathfrak{G}$ a minimal set of homogeneous generators of $M,$ and $\basis$ a $v$-basis of $M.$ Then the following statements are true.
\begin{enumerate}[(i)]
\item There exists a $v$-basis consisting of minimal generators of $M.$
\item If $b\in\basis$ is indispensable for a $\gen\in\mathfrak{G},$ then $[v\ast\gen]=0\in M/\langle b \rangle.$
\end{enumerate}
\end{corollary} 
\begin{proof}
Since $M$ is an interval module, each homogeneous element is a multiple of~some element of $\mathfrak{G}.$ Using \Cref{lemma:basis_exchange} iteratively, replace each element of $\basis$ by a~respective element of $\mathfrak{G}.$   If $b\in\basis$ is indispensable for $\gen,$ then $v\ast\gen$ is a multiple of~$b.$ 
\end{proof}
Note that the reverse of the second statement does not hold true.

\subsection{Algorithm for interval modules}\label{section algo}
The software {\tt Topcat}~\cite{topcat} of G\"afvert provides tools for calculating stable ranks for arbitrary persistence modules and, in principle, is based on searching through the entire $B(M,\delta)$ to look for an element with the smallest~$\beta_0$. This search can be done more efficiently for interval modules in the two-parameter case, and we
now present its algorithm.
An implementation in {\tt C++} of our algorithm  is made available at \href{https://github.com/recorb/shiftdim_stablerank}{https://github.com/recorb/shiftdim\_stablerank}.

Let $v=(v_1,v_2)\in\mathbb{R}_{>0}^2,$ 
$M$ be an interval module, and $\mathfrak{G}\coloneqq\{\gen_i\}_{i\in I\subseteq \Rnn}$ be the generating curve of $M.$ Denote \mbox{$\deg(\gen_i)=(a_i,b_i).$} W.l.o.g.\ assume that the generators are listed in the following total order: $a_i<a_j$ or $b_i>b_j$ whenever~$i<j.$
\begin{algo}\label{algo}
\begin{enumerate}[{\em Step 1.}]
\item[]
\item   Cluster $\{ \mathfrak{g}_i \}_{i\in I_1}$ in increasing order as follows: Set $I_1\coloneqq I,$ $k\coloneqq 1.$ \\
While there exists $i\in I_k$ such that $v\ast \gen_i \neq 0,$ do:
\begin{enumerate}
\item Set $i_k $ to be the maximum of the set\\ $  \left\{ i\in I_k \mid \deg(\gen_{i})\leq\deg({v\ast\gen}_{j}) \text{ for all }j\in I_k \text{ s.t. } j< i \text{ and } 
v\ast\gen_{j}\neq 0 \right\}.$
\item Define $I_{k+1} \coloneqq I_k \setminus \{ i  \in I_k \mid \deg(g_{i_k}) \leq \deg(v \ast \mathfrak{g}_{i}) \}.$
\item Replace $k$ by $k+1$ and iterate.
\end{enumerate}
\item Set $\mathfrak{H}\coloneqq \{\gen_{i_1},\ldots,\gen_{i_\ell}\}$ and $\ell\coloneqq  	\lvert \mathfrak{H} \rvert.$
\item Return $\ell$ and $\mathfrak{H}.$
\end{enumerate}
\end{algo}

Note that the computation of the $i_k$ and $I_{k+1}$ can be carried out geometrically: to obtain $i_k,$ we project from $(\inf\{a_i\mid i\in {I_k},v\ast\gen_i\neq0\},\sup\{b_i\mid i\in {I_k},v\ast\gen_i\neq0\})+v$ down to the generating curve of $M.$ 
For the construction of~$I_{k+1},$ we project to the right from $\deg(\gen_{i_k})-v$ to the generating curve and hit a coordinate $(a,b).$ We remove all generators whose second coordinate is not strictly greater than $b$ from $I_k$ to obtain $I_{k+1}.$ See~\Cref{figure:algorithm} for an illustration of this geometric procedure. 
\begin{figure}
	\centering
	\begin{minipage}[t]{0.42\textwidth}
		\scalebox{.75}{
			\begin{tikzpicture}[scale=0.45, transform shape]
			\foreach \r in {0,1,..., 14}
			\draw (\r,0) node[inner sep=.5pt,below=9pt,rectangle,fill=white] {\fontsize{13}{13}$\r$};
			\foreach \r in {0,1, 2,...,13}
			\draw (0,\r) node[inner sep=1pt,left=9pt,rectangle,fill=white] {\fontsize{13}{13}$\r$};
			\tikzstyle{every node}=[circle, draw, fill=niceblue!90, inner sep=1pt, minimum width=12pt] node{};
			\draw[black,->] (0,0)  -- (0,13.5) ;
			\draw[black,->]  (0,0)  -- (14.5,0) ;
			\foreach \r in {1,...,14} 
			\draw[black] (\r,-0.05) -- (\r,0.05);
			\foreach \r in {1,...,13} 
			\draw[black] (-0.05,\r) -- (0.05,\r);
			\draw[loosely dotted] (2.5,0) -- (2.5,13.5);
			\draw[name path = A]  (5.5,9) parabola bend (4.5,9) (2.65,13.5);
			\draw[name path = A2]  (2.65,13.5) -- (2.65,13.5);
			\draw[name path = B] (5.5,9) -- (5.5,6);
			\draw[name path = C] (5.5,6) -- (7,6);
			\draw[name path = D] (7,4) -- (7,6);
			\draw[name path = E] (7,4) -- (8.7,4);
			\draw[name path = F] (8.7,4) parabola bend (14.5,2)  (14.5,2) ;
			\draw[loosely dotted] (0,1.9) -- (14.5,1.9);
			\draw[name path = G] (10,11) parabola bend (11,4)  (11,4) ;
			\draw[name path = H] (11,4) -- (14.5,4);
			\draw[name path = I] (10,11) --  (9,11) ;
			\draw[name path = J] (9,11) --  (9,13.5) ;
			\begin{pgfonlayer}{bg}
			\tikzfillbetween[of=A and J] {\modulegray};
			\tikzfillbetween[of=A2 and J] {\modulegray};
			\tikzfillbetween[of=B and J] {\modulegray};
			\tikzfillbetween[of=B and I] {\modulegray};
			\tikzfillbetween[of=C and J] {\modulegray};
			\tikzfillbetween[of=E and J] {\modulegray};
			\tikzfillbetween[of=E and G] {\modulegray};
			\tikzfillbetween[of=E and I] {\modulegray};
			\tikzfillbetween[of=E and H] {\modulegray};
			\tikzfillbetween[of=F and H] {\modulegray};
			\tikzfillbetween[of=D and G] {\modulegray};
			\tikzfillbetween[of=D and G] {\modulegray};
			\end{pgfonlayer}
			\fill[\modulegray] (9,4) rectangle (10.9,4.01);
			\fill[\modulegray] (7,5) rectangle (9,11);
			\draw[niceblue!90,thick] (2.5,13.8) -- (4,13.8);
			\draw[niceblue!90,thick] (4,13.8) -- (4,9.33);
			\draw[niceblue!90,thick] (2.5,7.83) -- (4,9.33);
			\draw[niceblue!90,thick] (2.5,7.83) -- (5.5,7.83);
			\draw[niceblue!90,thick] (5.5,7.83) -- (7,9.33);
			\draw[niceblue!90,thick] (7,4) -- (7,9.33);
			\draw[niceblue!90,thick] (7,4) -- (5.5,2.5);
			\draw[niceblue!90,thick] (11.6,2.5) -- (5.5,2.5);
			\draw[niceblue!90,thick] (11.6,2.5) -- (13.1,4);
			\draw[niceblue!90,thick] (13.1,2.13) -- (13.1,4);
			\draw[niceblue!90,thick] (13.1,2.13) -- (11.6,0.63);
			\draw[niceblue!90,thick] (14.5,0.63) -- (11.6,0.63);
			\draw (4,9.3) node[fill=niceblue!90]{};
			\draw (7,4) node[fill=niceblue!90]{};
			\draw (13.1,2.12) node[fill=niceblue!90]{};
			\draw[decorate, decoration={brace}]  (2,7.83) -- node[draw=none,fill=none,left=0.6ex] {\huge$I_1$}  (2,13.8);
			\draw[decorate, decoration={brace}]  (2,2.5) -- node[draw=none,fill=none,left=0.6ex] {\huge$I_2$}  (2,7.83);
			\draw[decorate, decoration={brace}]  (2,1.9) -- node[draw=none,fill=none,left=0.6ex] {\huge$I_3$}  (2,2.5);
			\end{tikzpicture}  }
	\end{minipage}
	\begin{minipage}[t]{0.42\textwidth}
		\scalebox{.75}{
			\begin{tikzpicture}[scale=0.45, transform shape]
			\foreach \r in {0,1,..., 14}
			\draw (\r,0) node[inner sep=.5pt,below=9pt,rectangle,fill=white] {\fontsize{13}{13}$\r$};
			\foreach \r in {0,1, 2,...,13}
			\draw (0,\r) node[inner sep=1pt,left=9pt,rectangle,fill=white] {\fontsize{13}{13}$\r$};
			\tikzstyle{every node}=[circle, draw, fill=niceblue, inner sep=1pt, minimum width=12pt] node{};
			\draw[black,->] (0,0)  -- (0,13.5) ;
			\draw[black,->]  (0,0)  -- (14.5,0) ;
						\foreach \r in {1,...,14} 
			\draw[black] (\r,-0.05) -- (\r,0.05);
			\foreach \r in {1,...,13} 
			\draw[black] (-0.05,\r) -- (0.05,\r);
			\draw[loosely dotted] (2.5,0) -- (2.5,13.5);
			\draw[name path = A]  (5.5,9) parabola bend (4.5,9) (2.65,13.5);
			\draw[name path = A2]  (2.65,13.5) -- (2.65,13.5);
			\draw[name path = B] (5.5,9) -- (5.5,6);
			\draw[name path = C] (5.5,6) -- (7,6);
			\draw[name path = D] (7,4) -- (7,6);
			\draw[name path = E] (7,4) -- (8.7,4);
			\draw[name path = F] (8.7,4) parabola bend (14.5,2)  (14.5,2) ;
			\draw[loosely dotted] (0,1.9) -- (14.5,1.9);
			\draw[name path = G] (10,11) parabola bend (11,4)  (11,4) ;
			\draw[name path = H] (11,4) -- (14.5,4);
			\draw[name path = I] (10,11) --  (9,11) ;
			\draw[name path = J] (9,11) --  (9,13.5) ;
			\begin{pgfonlayer}{bg}
			\tikzfillbetween[of=A and J] {fill=\modulegray};
			\tikzfillbetween[of=A2 and J] {fill=\modulegray};
			\tikzfillbetween[of=B and J] {fill=\modulegray};
			\tikzfillbetween[of=B and I] {fill=\modulegray};
			\tikzfillbetween[of=C and J] {fill=\modulegray};
			\tikzfillbetween[of=E and J] {fill=\modulegray};
			\tikzfillbetween[of=E and G] {fill=\modulegray};
			\tikzfillbetween[of=E and I] {fill=\modulegray};
			\tikzfillbetween[of=E and H] {fill=\modulegray};
			\tikzfillbetween[of=F and H] {fill=\modulegray};
			\tikzfillbetween[of=D and G] {fill=\modulegray};
			\tikzfillbetween[of=D and G] {fill=\modulegray};
			\end{pgfonlayer}
			\fill[\modulegray] (9,4) rectangle (10.9,4.01);
			\fill[\modulegray] (7,5) rectangle (9,11);
			\draw[\goodred,thick] (14.8,1.9) -- (14.8,3.5);
			\draw[\goodred,thick] (9.47,3.5) -- (14.8,3.5);
			\draw[\goodred,thick] (9.47,3.5) -- (7.97,2);
			\draw[\goodred,thick] (7.97,2) -- (7.97,4);
			\draw[\goodred,thick] (7.97,4) -- (9.47,5.5);
			\draw[\goodred,thick] (7,5.5) -- (9.47,5.5);
			\draw[\goodred,thick] (7,5.5) -- (5.5,4);
			\draw[\goodred,thick] (5.5,6) -- (5.5,4);
			\draw[\goodred,thick] (5.5,9) -- (7,10.5);
			\draw[\goodred,thick] (7,10.5) -- (3.42,10.5);
			\draw[\goodred,thick] (1.92,9) -- (3.42,10.5);
			\draw[\goodred,thick] (1.92,9) -- (1.92,13.5);
			\draw (9.47,3.5) node[fill=\goodred]{};
			\draw (7,5.5) node[fill=\goodred]{};
			\draw (3.42,10.5) node[fill=\goodred]{};
			\draw[black,decorate, decoration={brace}]  (14.8,1.5) -- node[draw=none,fill=none,below=0.6ex] {\huge$I_1$}  (8,1.5);
			\draw[black,decorate, decoration={brace}]  (8,1.5) -- node[draw=none,fill=none,below=0.6ex] {\huge$I_2$}  (5.5,1.5);
			\draw[black,decorate, decoration={brace}]  (5.5,1.5) -- node[draw=none,fill=none,below=0.6ex] {\huge$I_3$}  (2.5,1.5);
			\end{tikzpicture}  }
	\end{minipage}
	\caption{{\em Left:} Constructing a $(1.5,1.5)$-basis, using~\Cref{algo}, clustering from above. {\em Right:} Using the opposite version from below.}
	\label{figure:algorithm}
\end{figure}
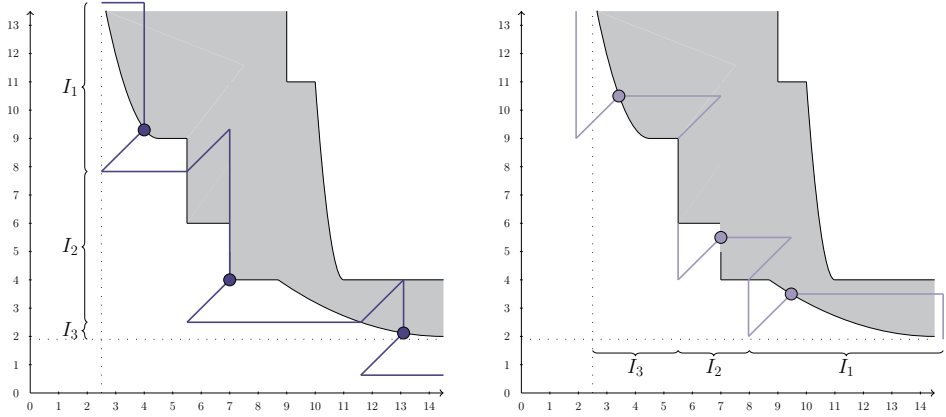

\begin{theorem}
\Cref{algo} computes $\dim_v(M)$ in at most $\lceil \frac{b_{i_1}-\inf\{b_i\mid i\in {I}\}}{v_2}\rceil$ iterations of Step~1.
\end{theorem}
\begin{proof}
First note that by~\Cref{corollary:basis_exchange} $(i),$ there is a subset of $\mathfrak{G}$ that is a $v$-basis of $M.$
Hence, a minimal subset $\mathfrak{H} $ of $ \mathfrak{G}$ such that $v\ast \gen_i\in  \langle \mathfrak{H}\rangle$ for all $i\in I$ is indeed a $v$-basis. If $\mathfrak{H}=\emptyset,$ the statement is clear. Hence, let $\ell >0.$
Let us first argue that $\ell$ is finite. Since $v_1\neq0,$ $i_1$ is well-defined. For all $0<k<\ell,$ $b_{i_k} - b_{i_{k+1}}\geq v_2 >0.$ Since $b_i\geq 0$ for all $i\in I,$ if follows that $\ell$ is finite; in fact, $\ell<\lceil \frac{b_{i_1}-\inf\{b_i\mid i\in {I}\}}{v_2}\rceil<\lceil \frac{b_{i_1}}{v_2}\rceil.$ 

We now prove that $\ell = \dim_v(M)$ and $\mathfrak{H}$ is a $v$-basis of $M.$ 
We have $v\ast (M/\langle \mathfrak{H} \rangle )=0.$ Hence, $\dim_v(M) \leq \ell.$ For the reverse direction, assume that $\ell>0$ and let $\basis\coloneqq\{\mathfrak{b}_j\}_{j\in J}$ be a $v$-basis of $M.$  Applying~\Cref{lemma:basis_exchange} iteratively, we can w.l.o.g.\ assume that for each~$j,$ $\mathfrak{b}_j=\gen_i$ for some $i\in I.$ 
Hence, w.l.o.g.\ $\basis$ is a subset of $\mathfrak{G}$ that might differ from~$\mathfrak{H}.$  
However, $\basis$ clusters $I$ into $\dim_v(M)$ subsets;  the  cluster labeled by $j\in J$ consists of all $i\in I$ for which $\deg(\mathfrak{b}_j) \leq (\deg{v\ast \gen_i}).$ If two elements of $I$ both are elements of one such cluster, then each element in between must be in the same cluster as well.
By construction, for all $1\leq k \leq \ell$ there is no $i\in I\setminus({\bigcup_{1\leq j\leq k-1}{I_j}})$ such that there exists $\tilde{I_k}\supsetneq I_k$ with $\tilde{I_k}\cap(\bigcup_{1\leq j\leq k-1}{I_j})=\emptyset$ and $\langle \{v\ast\gen_j \mid j\in\tilde{I}_k\} \rangle \subseteq \langle\gen_i\rangle.$ 
Hence, the cluster given by $\basis$ which contains the smallest elements of $I$ must be a subset of~$I_1.$ By construction of $I_2,$ the second cluster given by $\basis$ can not be a proper superset of $I_1\cup I_2.$ Iteratively, we get that the cardinality of $\basis$ is at least $\ell.$
\end{proof}
Replacing the total order on the generators by the opposite total order yields another variant of the algorithm computing~$\dim_v(M).$ This yields the upper bound $\ell\leq \lceil \frac{a_{i_1}-\inf\{a_i\mid i\in {I}\}}{v_1}\rceil \leq \lceil \frac{a_{i_1}}{v_1}\rceil.$
The algorithm can be extended to the case $v_1=0,$ $v_2\neq0$ if $\sup\{b_i\mid i\in {I}\}<\infty$  and---using the opposite total order---to the case $v_1\neq0,$ $v_2=0$ if $\sup\{a_i\mid i\in {I}\}<\infty.$ If in these cases the suprema are not finite, $\dim_v(M)=\infty.$

For the complexity of the algorithm, it remains to investigate the complexity of Steps 1(a) and 1(b). 
By the geometric construction described above, we need to compute infima, suprema, add and subtract a vector, project to a curve, and check if certain segments of a curve get shifted beyond another curve. In general, computing the shift-dimension is NP-hard, as was mentioned in~\Cref{remark:L}. We are now going to prove that the computation is linear for tame interval modules.

\begin{corollary}\label{cor:lineartime}
If $M$ is a finitely presented interval module in the bivariate case, then $\dim_v(M)$ can be computed in $O(n)$ time, where $n\coloneqq \beta_0(M).$
\end{corollary}
\begin{proof}
Each iteration of Step 1 boils down to carrying out two additions, determining three minima of a finite array of sorted data, and deciding whether ${v\ast\gen}_{j}=0$ only for a subset of minimal homogeneous generators. Deciding whether ${v\ast\gen}_{j}=0$ can again be done by finding the minimum of a finite array of sorted data. The total number of the latter operations can naively bounded by $\beta_0(M)$ and can hence be done as a preprocessing step of $\beta_0(M)$ operations. All the individual operations can be done in $O(1)$ time. 
Running the algorithm from above or below, the number of iterations of Step 1 is bounded by $\lceil \frac{a_{i_1}-\min\{a_i\mid i\in {I}\}}{v_1}\rceil$ or $\lceil \frac{b_{i_1}-\min\{b_i\mid i\in {I}\}}{v_2}\rceil,$ respectively. Since $\beta_0(M)$ is finite, it is an additional upper bound for the number of iterations of Step~1. 
\end{proof}

For some interval modules, such as monomial ideals, the module generated by $\mathfrak{H}$ has an illustrative description: it is a staircase with the minimal possible number of steps that fits between $M$ and $v*M.$  
This investigation might give additional ideas for a proof of an analogous algorithm in the setting of more than two parameters. Such a generalization would presumably increase the geometric complexity of the algorithm. 

The insights of this algorithm in combination with a better structural understanding of the examples in \Cref{section:locus} might give rise to an efficient way to compute the shift-dimension of direct sums of interval modules.

\section{The shift-dimension of multigraded $\field[x_1,\ldots,x_\param]$-modules}\label{section algebraic}
In this section, we explain how the shift-dimension for tame persistence modules translates to the monoid $G=\mathbb{N}^r.$ Identifying the monoid ring $\field[\mathbb{N}^r]$ with the polynomial ring $\field[x_1,\ldots,x_r]$ and hence translating $v\ast (\bullet)$ to multiplication by the monomial $x_1^{v_1}\cdots x_r^{v_r},$ the study of $\Mod^{\gr}(\field[G])$ is tantamount to the study of multigraded modules over the multivariate polynomial ring. Those modules play a  fundamental role in algebraic geometry.
We denote by $S\coloneqq \field[x_1,\ldots,x_\param]$ the polynomial ring in $\param$ variables over a field $\field$   with the standard $\NN^{\param}$-grading, i.e., the degree of $x_i$ is the $i$-th standard unit vector $\mathbf{e}_i \in \NN^{\param}.$ Motivated by the stabilization of $\beta_0$ with respect to the metric arising from the standard contour in the direction of a vector, we introduce the following invariant of multigraded $S$-modules.

\begin{definition}
	Let $M$ be an $\NN^\param$-graded $S$-module and $v\in \NN^\param.$ Elements $m_1,\ldots,m_k$ of $M$ {\em $v$-generate $M$}, if 
	$$ x_1^{v_1}\cdots x_\param^{v_\param} \cdot  M \, \subseteq \, \langle m_1,\ldots,m_k \rangle .$$ 
	A {\em $v$-basis of $M$} is a collection $\{m_1,\ldots,m_k\}$ of  elements  of $M$ that \mbox{$v$-generates} $M$ and for which $k$ is smallest possible. In this case,  we call $k$ the {\em $v$-dimension} of $M$ and denote it by $\dim_v (M).$
\end{definition}
Again, for finitely generated persistence modules, we stick to {\em homogeneous} generators and \mbox{$v$-bases}.  For fixed $v\in \NN^\param,$ we abbreviate $\dim_{nv}$ by $\dim_n$ and refer to it as {\em $n$-dimension}, and similarly for ``$n$-basis'' and ``$n$-generators''.
\begin{remark}
This definition generalizes to arbitrary rings $R.$ Let $M$ be an $R$-module and $r\in R.$ Elements $\{m_1,\ldots,m_k\}$ of $M$ {\em $r$-generate} $M$ if $rM\subseteq \langle m_1,\ldots,m_k \rangle.$
\end{remark}
Note that $\dim_0(M)$ is the minimal number of (homogeneous) generators of $M.$ Moreover, $\dim_n(M)$ is zero if and only if $(x_1^{v_1}\cdots x_\param^{v_\param})^n M$ is the zero-module.  

\begin{remark} By Hilbert's syzygy theorem, every finitely generated \mbox{$\NN^\param$-graded}\linebreak \mbox{$S$-module} has a minimal free resolution $F_{\bullet}$ of length at most $\param$ (see \cite[Proposition 8.18]{MillSturm}). The rank of $F_i$ is the {\em $i$-th total multigraded Betti number} of the module. It would be intriguing to stabilize higher total multigraded Betti numbers as well.
\end{remark}

The following proposition explains how to modify tame persistence modules to obtain a module in which every element can be extended to an $n$-basis. This construction works for the monoid ring $\field[\NN^r],$ but not for $\field[\Rnn^r].$
\begin{proposition}\label{prop:killnonbasiselementsN}
	Let $M$ be a finitely generated $\NN^\param$-graded $S$-module and $n\in \NN.$ Choose a non-zero element $m_1\in M$ that is not contained in any $n$-basis of $M$ and consider the quotient module $M/\langle m_1 \rangle.$ Then choose a non-zero element $[m_2]\in M/\langle m_1 \rangle$ that is not contained in any $n$-basis of $M/\langle m_1 \rangle,$ and repeat this process. There exists a natural number $\ell$ such that after $\ell$ iterations
$$ M \longrightarrow M/\langle m_1 \rangle \longrightarrow M/\langle m_1,m_2\rangle \longrightarrow \cdots \longrightarrow M/ \langle m_1,\ldots,m_\ell \rangle \, \eqqcolon\, \widetilde{M},$$ one arrives at a module $\widetilde{M}$ of the same $n$-dimension as $M$ in which every non-zero element can be extended to an $n$-basis of $\widetilde{M}.$ 
\end{proposition}
\begin{proof}
 Since $M$ is finitely generated and $S$ is Noetherian, the procedure described in the proposition stabilizes after a finite number of steps.
\end{proof}

Note that the module $\widetilde{M}$ can in general {\em not} be obtained by successively quotienting out non-$n$-basis-elements of $M$ itself.  For instance, consider $M=\langle x_1,x_2\rangle$ and $v=\mathbbm{1}$ the all-one vector. The element $m_1\coloneqq x_1^5x_2^3$ is not in any $2$-basis of $M$ and $m_2\coloneqq x_1^3x_2^5$ is not in a $2$-basis of~$M,$ but $\{[m_2]\}$ is a $2$-basis of the module~$M/\langle m_1\rangle.$

\begin{remark}  Recall that for two ideals $I,J$ in a ring $R,$  the {\em ideal quotient} of $I$ by $J$
\mbox{$ (I:J)  \coloneqq \{r\in R \mid rJ \subseteq I \}  \triangleleft  R$}
is itself an ideal in $R.$   Let $M=I\triangleleft \field[x_1,\ldots,x_\param]$ be a homogeneous ideal, $v\in \mathbb{N}^r.$
Then for all $n\in \mathbb{N},$ the $n$-dimension of the ideal quotient $(I\colon (x_1^{v_1}\cdots x_\param^{v_\param})^n)$ bounds the minimal number of generators of $I$ from below.
\end{remark}

\begin{example}
	Let  $M$ be the monomial ideal $\langle x_1^2,x_2^3,x_3^5\rangle \triangleleft \CC [x_1,x_2,x_3]$ and $v=\mathbbm{1}.$ For $n\in \{0,1\},$ the three elements $x_1^2,x_2^3,x_3^5$ $n$-generate $M.$ For $n\geq 2,$ $M$ is $n$-generated by the single element $x_1^2.$ Therefore, the sequence $(\dim_n(M))_{n\in \NN}$ is $3,3,1,1,1,\ldots $
\end{example}

\begin{example}
	Let $I=\langle x_1x_2^4,x_1^3x_2^2,x_1^5x_2\rangle \triangleleft \field [x_1,x_2]$ and $v=\mathbbm{1}.$ Then $I^2$ is minimally generated by the $5$ elements $x_1^2x_2^8,x_1^4x_2^6,x_1^6x_2^4,x_1^8x_2^3,x_1^{10}x_2^2.$ The sequences of $n$-dimensions are $(\dim_n(I))_{n\in \mathbb{N}}=3,3,1,1,\ldots $ and $(\dim_n(I^2))_{n \in \mathbb{N}}=5,5,1,1,\ldots $
\end{example}

\begin{example}[Free multigraded module]
	Let $F=S(-a_1,\ldots,-a_\param)\cong S \cdot x_1^{a_1}\cdots x_\param^{a_\param}$ be a free multigraded module, where $a_1,\ldots,a_\param\in \NN.$
	One reads that $F$ is $n$-generated by the single element $x_1^{a_1+nv_1}\cdots x_\param^{a_\param +nv_\param}.$ Therefore $(\dim_n(F))_{n\in \NN}=1,1,1,\ldots$
\end{example}

\begin{example}[Quotient of two monomial ideals]\label{example monomials}
Let $M=\langle x_1^3x_2,x_1x_2^3\rangle / \langle x_1^4x_2^4 \rangle$ and $v=\mathbbm{1}.$ Then $(\dim_n(M))_{n\in \NN}=2,2,1,0,0,\ldots$ For a visualization, see~\Cref{figure monomials}.
\end{example}
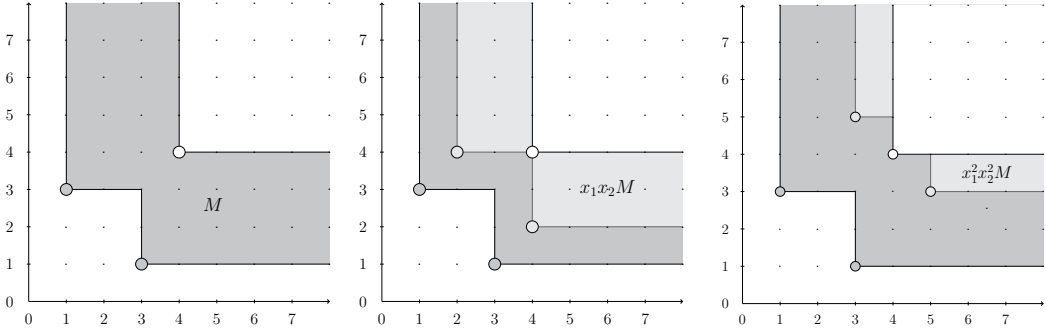
\begin{figure}
	\begin{minipage}[t]{0.30\textwidth}
		\scalebox{.45}{
			\begin{tikzpicture}[scale=1.1, transform shape]
			\tikzstyle{grid lines}=[lightgray,line width=0]
			\foreach \r in {0,..., 7}
			\draw (\r,0) node[inner sep=.5pt,below=9pt,rectangle,fill=white] {\fontsize{13}{13}$\r$};
			\foreach \r in {0,...,7}
			\draw (0,\r) node[inner sep=1pt,left=9pt,rectangle,fill=white] {\fontsize{13}{13}$\r$};
			\tikzstyle{every node}=[circle, draw, fill=black, inner sep=1pt, minimum width=9pt] node{};
			\fill[\modulegray] (1,3) rectangle (8,8);
			\fill[\modulegray] (3,1) rectangle (8,8);
			\fill[white] (4,4) rectangle (8,8);
			\foreach \r in {1,...,7} 
			\draw[black] (\r,-0.05) -- (\r,0.05);
			\foreach \r in {1,...,7} 
			\draw[black] (-0.05,\r) -- (0.05,\r);
			\foreach \r in {1,...,8}
			\foreach \s in {1,...,8}
			\node[fill=none,draw=none] at (\r,\s) {.};
			\draw[black] (1,3)  -- (1,8) ;
			\draw[black] (1,3)  -- (3,3) ;
			\draw[black] (3,1)  -- (3,3) ;
			\draw[black] (3,1)  -- (8,1) ;
			\draw[black] (4,4)  -- (4,8) ;
			\draw[black] (4,4)  -- (8,4) ;
			\draw[black,->] (0,0)  -- (0,8) ;
			\draw[black,->]  (0,0)  -- (8,0) ;		
			\node[rectangle,draw=none,fill=none,label=\Large$M$] at (4.9,2.15) {};
			\draw (1,3) node[fill=\modulegray]{};
			\draw (3,1) node[fill=\modulegray]{};
			\draw (4,4) node[fill=white]{};
			\end{tikzpicture}}
	\end{minipage}\hspace*{2mm}
	\begin{minipage}[t]{0.30\textwidth}
		\scalebox{.45}{
			\begin{tikzpicture}[scale=1.1, transform shape]
			\tikzstyle{grid lines}=[lightgray,line width=0]
			\foreach \r in {0,..., 7}
			\draw (\r,0) node[inner sep=.5pt,below=9pt,rectangle,fill=white] {\fontsize{13}{13}$\r$};
			\foreach \r in {0,...,7}
			\draw (0,\r) node[inner sep=1pt,left=9pt,rectangle,fill=white] {\fontsize{13}{13}$\r$};
			\tikzstyle{every node}=[circle, draw, fill=black, inner sep=1pt, minimum width=9pt] node{};
			\fill[\modulegray] (1,3) rectangle (8,8);
			\fill[\modulegray] (3,1) rectangle (8,8);
			\fill[\shiftgray] (4,2) rectangle (8,8);
			\fill[\shiftgray] (2,4) rectangle (8,8);
			\fill[white] (4,4) rectangle (8,8);
			\draw[black] (1,3)  -- (1,8) ;
			\draw[black] (1,3)  -- (3,3) ;
			\draw[black] (3,1)  -- (3,3) ;
			\draw[black] (3,1)  -- (8,1) ;
			\draw[black] (4,4)  -- (4,8) ;
			\draw[black] (4,4)  -- (8,4) ;
			\draw[black!60] (2,4)  -- (2,8) ;
			\draw[black!60] (2,4)  -- (4,4) ;
			\draw[black!60] (4,2)  -- (4,4) ;
			\draw[black!60] (4,2)  -- (8,2) ;
			\draw[black,->] (0,0)  -- (0,8) ;
			\draw[black,->]  (0,0)  -- (8,0) ;		
			\foreach \r in {1,...,7} 
			\draw[black] (\r,-0.05) -- (\r,0.05);
			\foreach \r in {1,...,7} 
			\draw[black] (-0.05,\r) -- (0.05,\r);
			\foreach \r in {1,...,8}
			\foreach \s in {1,...,8}
			\node[fill=none,draw=none] at (\r,\s) {.};
			\draw (1,3) node[fill=\modulegray]{};
			\draw (3,1) node[fill=\modulegray]{};
			\draw (4,4) node[fill=white]{};
			\draw (4,2) node[fill=\shiftgray]{};
			\draw (2,4) node[fill=\shiftgray]{};
			\node[rectangle,draw=none,fill=none,label=\Large$x_1x_2M$] at (6,2.15) {};
			\end{tikzpicture}}
	\end{minipage} \hspace*{2mm}
	\begin{minipage}[t]{0.30\textwidth}
		\scalebox{.45}{
			\begin{tikzpicture}[scale=1.1] 
			\tikzstyle{grid lines}=[lightgray,line width=0]
			\foreach \r in {0,..., 7}
			\draw (\r,0) node[inner sep=.5pt,below=9pt,rectangle,fill=white] {\fontsize{13}{13}$\r$};
			\foreach \r in {0,..., 7}
			\draw (0,\r) node[inner sep=1pt,left=9pt,rectangle,fill=white] {\fontsize{13}{13}$\r$};
			\tikzstyle{every node}=[circle, draw, fill=black, inner sep=1pt, minimum width=8pt] node{};
			\fill[\modulegray] (1,3) rectangle (8,8);
			\fill[\modulegray] (3,1) rectangle (8,8);
			\fill[\shiftgray] (3,5) rectangle (8,8);
			\fill[\shiftgray] (5,3) rectangle (8,8);
			\fill[white] (4,4) rectangle (8,8);
			\draw[black] (1,3)  -- (1,8) ;
			\draw[black] (1,3)  -- (3,3) ;
			\draw[black] (3,1)  -- (3,3) ;
			\draw[black] (3,1)  -- (8,1) ;
			\draw[black] (4,4)  -- (4,8) ;
			\draw[black] (4,4)  -- (8,4) ;
			\draw[black!60] (3,5)  -- (3,8) ;
			\draw[black!60] (3,5)  -- (4,5) ;
			\draw[black!60] (5,3)  -- (5,4) ;
			\draw[black!60] (5,3)  -- (8,3) ;
			\draw[black,->] (0,0)  -- (0,8) ;
			\draw[black,->]  (0,0)  -- (8,0) ;		
			\foreach \r in {1,...,7} 
			\draw[black] (\r,-0.05) -- (\r,0.05);
			\foreach \r in {1,...,7} 
			\draw[black] (-0.05,\r) -- (0.05,\r);
			\foreach \r in {1,...,8}
			\foreach \s in {1,...,8}
			\node[fill=none,draw=none] at (\r,\s) {.};
			\draw (1,3) node[fill=\modulegray]{};
			\draw (3,1) node[fill=\modulegray]{};
			\draw (4,4) node[fill=white]{};
			\draw (5,3) node[fill=\shiftgray]{};
			\draw (3,5) node[fill=\shiftgray]{};
			\node[draw=none,fill=none,label=\Large$x_1^2x_2^2M$] at (6.5,2.55) {.};
			\end{tikzpicture}}
	\end{minipage}
	\caption{Visualization of $M,x_1x_2M,$ and $x_1^2x_2^2M$ for $M$ from \Cref{example monomials}} 
	\label{figure monomials}
\end{figure}

For $r=1,$ the shift-dimension is additive. For $r>1,$ this does in general {\em not} hold true, as the following counterexamples for $\param=2$ and $v=\mathbbm{1}$ demonstrate.
\begin{example}[Revisiting \Cref{counterexampleadd}]
Let $M=\langle x_1  \rangle / \langle x_1 x_2^2 \rangle$ and $N = \langle x_2 \rangle / \langle x_1 ^2 x_2 \rangle.$ Then $x_1 x_2 (M\oplus N)$  is contained in the submodule $\langle (x_1 x_2,x_1 x_2)\rangle$ of $M\oplus N.$
	To see that, observe that $x_1 (x_1 x_2,x_1 x_2)=x_1 x_2(x_1 ,0)$ and $x_2(x_1 x_2,x_1 x_2)=x_1 x_2(0,x_2)$ in  $M \oplus N.$	
	It follows that $\dim_1(M\oplus N)=1 \neq 1+1= \dim_1(M)+\dim_1(N).$
\end{example}

\begin{example}
	Let $M_1=\langle x_1 ^2\rangle / \langle x_1 ^6,x_1 ^2x_2^4\rangle,$ $M_2=\langle x_1 x_2 \rangle / \langle x_1 ^5x_2,x_1 x_2^5\rangle ,$ and\linebreak$M_3=\langle x_2^2 \rangle / \langle x_1 ^4 x_2^2,x_2^6 \rangle.$ Then $\dim_3(M_i)=1$ for $i=1,2,3$ but $\dim_3(M_1\oplus M_2 \oplus M_3)=1,$ since the singleton $\{(x_1 ^3x_2^3,x_1 ^3x_2^3,x_1 ^3x_2^3)\}$ is a $3$-basis of the module $M_1\oplus M_2 \oplus M_3.$
\end{example}

We now present some cases for which additivity of the shift-dimension does hold true. They all are immediate consequences of statements presented in \Cref{section properties}. Here, we rephrase the statements for the discretized setting and omit the proofs.

\begin{lemma}
	If $F$ and $G$ are two free multigraded modules of finite rank, then $\dim_v(F\oplus G)=\dim_v (F)+\dim_v(G)$ for all $v\in \NN^r.$ \hfill \qed
\end{lemma}

\begin{proposition}
	Let $M=S/\langle x_1^{a_1}\cdots x_\param^{a_\param}\rangle $ and $F$ a free multigraded module of rank one. Then $\dim_v(M\oplus F)=\dim_v(M)+\dim_v(F)$ for all $v.$ \hfill \qed
\end{proposition}

\begin{proposition}
	Let $\dim_n(M)\leq 1$ and $F$ a free multigraded module of rank $k.$ Then $\dim_v (M\oplus F)=\dim_v(M)+k.$ \hfill \qed
\end{proposition}

\begin{theorem}
	Let $0 \longrightarrow M \stackrel{\varphi}{\longrightarrow} L \stackrel{\psi}{\longrightarrow} N \to 0$ be a short exact sequence of $\NN^\param$-graded \mbox{$S$-modules}. Then for all $v,w\in \NN^r,$ the following two inequalities hold:
	\begin{enumerate}[(i)]
		\item  $\dim_{v+w}(L)\leq \dim_v(M)+\dim_w(N),$
		\item $\dim_v(L) \leq \dim_v (N) + \beta_0(M).$  \hfill \qed
	\end{enumerate} 
\end{theorem} 
\noindent If $\dim_v(M) = \beta_0 (M)$ or $\beta_0(M)=1,$ it follows that $\dim_{v}(L)\leq \dim_v(M)+\dim_v(N).$ 

\section{Outlook to future work}\label{section outlook}
Our article outlines several pathways for future research. 
On the one hand, it would be intriguing to stabilize discrete invariants other than the zeroth total multigraded Betti number, such as the rank of the highest syzygy module. The following question immediately arises: {\em Is the rank of the highest syzygy module  ``naturally'' linked to a noise system---similar to the zeroth total multigraded Betti number being linked to the noise system in the direction of a vector in the sense of \Cref{def simple}?} A further challenging question is whether the stabilized Euler characteristic is the alternating sum of the stabilization of the ranks.

On the other hand, it would be worthwhile to investigate the stabilization of $\beta_0$ with respect to pseudometrics other than the one arising from the standard contour. For that, a first step is the description of the shift of the module. As described, for the standard contour in the direction of a vector, this is the module multiplied by the corresponding monomial. For other contours, there is no such algebraic description~yet.

Another interesting direction would be to investigate the shift-dimension for (not necessarily graded) modules over arbitrary rings and to develop a geometric intuition.

We believe that the shift-dimension allows for the extraction of finer and new information of data. This makes it valuable for applications in the medical and life sciences, among others.
In order to compute the shift-dimension of multipersistence modules of actual data arising in the sciences efficiently, 
one needs to extend the linear-time algorithm presented in this article to modules other than interval modules, such as direct sums of interval modules. Those indeed arise from data~\cite{escolar2016persistence} and approximate persistence modules in the sense of~\cite{asashiba2018interval},~\cite{botnan2022signed}, and~\cite{chacholski2022effective}. 
Using our algorithm, evaluating the shift-dimension summand-wise is already possible for this class of modules. It would be worthwhile to investigate if this additive version is stable, as was suggested to us by Ulrich Bauer. Furthermore, it would be interesting to find bounds for $\err_{v,p}$ and to determine subclasses of modules for which this error is zero.

In the one-parameter case, two persistence modules $M$ and $N$ are known to be isomorphic if and only if  $\widehat{\beta_0}_C(M)$ and $ \widehat{\beta_0}_C(N)$ coincide for each persistence contour~$C.$ At present, it is not known if that statement holds true for the multiparameter case as well. As a first step in this direction, one might  investigate whether one can distinguish non-isomorphic multigraded modules in the following sense: given two non-isomorphic modules $M,N,$ can one find a noise system $(\mathcal{V}_{\varepsilon})_{\varepsilon\in \mathbb{R}_{\geq 0}}$ such that $M\in \mathcal{V}_0$ but $N\notin \mathcal{V}_0$?


\bigskip\bigskip


\newpage

\begin{thebibliography}{10}
	
	\bibitem{asashiba2018interval}
	H.~Asashiba, M.~Buchet, E.~G. Escolar, K.~Nakashima, and M.~Yoshiwaki.
	\newblock On interval decomposability of 2{D} persistence modules.
	\newblock {\em Comput. Geom.}, 105--106:101879, 2022.
	
	\bibitem{Biasotti2011}
	S.~Biasotti, A.~Cerri, P.~Frosini, and D.~Giorgi.
	\newblock A new algorithm for computing the 2-dimensional matching distance
	between size functions.
	\newblock {\em Pattern Recognition Letters}, 32(14):1735--1746, 2011.
	
	\bibitem{BC20}
	M.~B. Botnan and W.~Crawley-Boevey.
	\newblock Decomposition of persistence modules.
	\newblock {\em Proc. Amer. Math. Soc.}, 148:4581--4596, 2020.
	
	\bibitem{BCM20}
	M.~B. Botnan, J.~Curry, and E.~Munch.
	\newblock A relative theory of interleavings.
	\newblock Preprint arXiv:2004.14286, 2020.
	
	\bibitem{botnan2022signed}
	M.~B. Botnan, S.~Oppermann, and S.~Oudot.
	\newblock Signed barcodes for multi-parameter persistence via rank
	decompositions.
	\newblock In {\em Proceedings of the 38th International Symposium on
		Computational Geometry (SoCG 2022)}. Schloss Dagstuhl-Leibniz-Zentrum f{\"u}r
	Informatik, 2022.
	
	\bibitem{superlinear}
	P.~Bubenik, V.~de~Silva, and J.~Scott.
	\newblock Metrics for generalized persistence modules.
	\newblock {\em Found. Comput. Math}, 15:1501--1531, 2015.
	
	\bibitem{buchet2020every}
	M.~Buchet and E.~G. Escolar.
	\newblock Every 1{D} persistence module is a restriction of some indecomposable
	2{D} persistence module.
	\newblock {\em J. Appl. and Comput. Topology}, 4:387–424, 2020.
	
	\bibitem{CZ09}
	G.~Carlsson and A.~Zomorodian.
	\newblock The theory of multidimensional persistence.
	\newblock {\em Discrete Comput. Geom.}, 42:71--93, 2009.
	
	\bibitem{chacholski2022effective}
	W.~Chach\'{o}lski, A.~Guidolin, I.~Ren, M.~Scolamiero, and F.~Tombari.
	\newblock Effective computation of relative homological invariants for functors
	over posets.
	\newblock Preprint arXiv:2209.05923, 2022.
	
	\bibitem{CR20}
	W.~Chach\'olski and H.~Riihim\"aki.
	\newblock Metrics and stabilization in one parameter persistence.
	\newblock {\em SIAM J. Appl. Algebra Geom.}, 4(1):69--98, 2020.
	
	\bibitem{CSV17}
	W.~Chach\'{o}lski, M.~Scolamiero, and F.~Vaccarino.
	\newblock Combinatorial presentation of multidimensional persistent homology.
	\newblock {\em J. Pure Appl. Algebra}, 221:1055--1075, 2017.
	
	\bibitem{stability}
	D.~Cohen-Steiner, H.~Edelsbrunner, and J.~Harer.
	\newblock Stability of persistence diagrams.
	\newblock {\em Discrete Comput. Geom.}, 37:103--120, 2007.
	
	\bibitem{corbet2018representation}
	R.~Corbet and M.~Kerber.
	\newblock The representation theorem of persistence revisited and generalized.
	\newblock {\em J. Appl. and Comput. Topology}, 2(1-2):1--31, 2018.
	
	\bibitem{EdelsbrunnerLetscherZomorodian2002}
	H.~Edelsbrunner, D.~Letscher, and A.~J. Zomorodian.
	\newblock Topological persistence and simplification.
	\newblock {\em Discrete Comput. Geom.}, 28:511--533, 2002.
	
	\bibitem{escolar2016persistence}
	E.~G. Escolar and Y.~Hiraoka.
	\newblock Persistence modules on commutative ladders of finite type.
	\newblock {\em Discrete Comput. Geom.}, 55(1):100--157, 2016.
	
	\bibitem{topcat}
	O.~G\"afvert.
	\newblock Topcat - a library for multiparameter persistence.
	\newblock Available at \url{https://github.com/olivergafvert/topcat}.
	
	\bibitem{GC}
	O.~G\"afvert and W.~Chach\'olski.
	\newblock Stable invariants for multidimensional persistence.
	\newblock Preprint arXiv:1703.03632, 2017.
	
	\bibitem{Ghr08}
	R.~Ghrist.
	\newblock Barcodes: the persistent topology of data.
	\newblock {\em Bull. Amer. Math. Soc.}, 45:61--75, 2008.
	
	\bibitem{HOST17}
	H.~Harrington, N.~Otter, H.~Schenck, and U.~Tillmann.
	\newblock Stratifying multiparameter persistent homology.
	\newblock {\em SIAM J. Appl. Algebra Geom.}, 3(3):439--471, 2017.
	
	\bibitem{Keller2018}
	B.~Keller, M.~Lesnick, and T.~L. Willke.
	\newblock {Persistent Homology for Virtual Screening}.
	\newblock Preprint ChemRxiv, 2018.
	
	\bibitem{KerberLesnickOudot2019}
	M.~Kerber, M.~Lesnick, and S.~Oudot.
	\newblock Exact computation of the matching distance on 2-parameter persistence
	modules.
	\newblock In {\em Proceedings of the 35th International Symposium on
		Computational Geometry (SoCG 2019)}, pages 46:1--46:15, 2019.
	
	\bibitem{kerber2021fast}
	M.~Kerber and A.~Rolle.
	\newblock Fast minimal presentations of bi-graded persistence modules.
	\newblock In {\em 2021 Proceedings of the Symposium on Algorithm Engineering
		and Experiments (ALENEX)}, pages 207--220. SIAM, 2021.
	
	\bibitem{lesnick2015theory}
	M.~Lesnick.
	\newblock The theory of the interleaving distance on multidimensional
	persistence modules.
	\newblock {\em Found. Comput. Math}, 15(3):613--650, 2015.
	
	\bibitem{lesnick2019computing}
	M.~Lesnick and M.~Wright.
	\newblock Computing minimal presentations and {B}etti numbers of 2-parameter
	persistent homology.
	\newblock {\em SIAM J. Appl. Algebra Geom.}, 6(2), 2022.
	
	\bibitem{MillSturm}
	E.~Miller and B.~Sturmfels.
	\newblock {\em Combinatorial {C}ommutative {A}lgebra}, volume 227 of {\em
		Graduate Texts in Mathematics}.
	\newblock Springer-Verlag, New York, 2005.
	
	\bibitem{rank3}
	R.~Peeters.
	\newblock Orthogonal representations over finite fields and the chromatic
	number of graphs.
	\newblock {\em Combinatorica}, 16:417--431, 1996.
	
	\bibitem{noise}
	M.~Scolamiero, W.~Chach\'{o}lski, A.~Lundman, R.~Ramanujam, and S.~\"{O}berg.
	\newblock Multidimensional persistence and noise.
	\newblock {\em Found. Comput. Math}, 17(6):1367--1406, 2017.
	
	\bibitem{rivet}
	{The {\tt RIVET} Developers}.
	\newblock {\tt RIVET}: {A} program for the visualization and analysis of
	two-parameter persistent homology.
	\newblock Available at \url{https://github.com/rivetTDA/rivet}, 2020.
	
	\bibitem{thomasdiss}
	A.~Thomas.
	\newblock {\em Invariants and Metrics for Multiparameter Persistent Homology}.
	\newblock PhD thesis, Duke University, 2019.
	
	\bibitem{ZomorodianCarlsson2005}
	A.~Zomorodian and G.~Carlsson.
	\newblock Computing persistent homology.
	\newblock {\em Discrete Comput. Geom.}, 33(2):249--274, 2005.
	
\end{thebibliography}
\end{document}